\documentclass{article}
\usepackage{geometry}

\usepackage{amsmath,amsfonts,amssymb,mathtools,amsthm}
\usepackage{bm}
\usepackage{enumitem}
\usepackage{tabularx}
\usepackage{array}
\usepackage{xcolor}
\usepackage{graphicx}
\usepackage{hyperref}

\usepackage{algorithm}
\usepackage{algorithmic}

\newtheorem{theorem}{Theorem}[section]
\newtheorem{definition}{Definition}[section]
\newtheorem{lemma}{Lemma}[section]
\newtheorem{corollary}{Corollary}[section]
\newtheorem{proposition}{Proposition}[section]

\newtheorem{remark}{Remark}

\newtheorem{example}{Example}

\newcommand{\R}{\mathbb R}

\DeclareMathOperator{\rank}{rank}
\DeclareMathOperator{\range}{range}

\newcommand{\risk}{\mathfrak R_M}
\newcommand{\riskhat}{\widehat{\mathfrak R}_M}
\newcommand{\E}{\mathbb E}
\DeclareMathOperator{\tr}{tr}

\newcommand{\norm}[1]{\left\lVert #1 \right\rVert}

\newcommand{\email}[1]{\href{mailto:#1}{#1}}
\newcommand{\rev}[1]{{#1}}

\title{Revisiting CUR Perturbation Analysis: A Local Tangent-Space Expansion}

\author{
Longxiu Huang\thanks{Department of CMSE, Department of Mathematics, Michigan State  University, East Lansing, MI 48824
  (\email{huangl3@msu.edu}).}
}
\date{}

\begin{document}

\maketitle

\begin{abstract}
CUR decompositions approximate a matrix using selected columns, rows, and their intersection. Classical CUR theory provides exactness results for low-rank matrices and perturbation bounds controlled by the size of the noise. In this work we develop a local perturbation expansion for a fixed-index rank-truncated CUR map near an admissible rank-\(r\) matrix. We show that the Fr\'echet derivative of the rank-truncated CUR map is a sampling-induced oblique tangent-space projector determined by the selected rows and columns.
Exact recovery implies that this projector annihilates the perturbation. The converse holds when the smallest nonzero singular value of the clean intersection exceeds the spectral norm of its perturbation.
For bounded, mean-zero isotropic noise, we derive the leading mean-square recovery error explicitly in terms of the sampled row and column subspaces. Separate A-optimal designs minimize it when the two sets can be chosen independently at fixed budgets.
Truncating the CUR product can give arbitrarily large first-order amplification even along exact-recovery directions for intersection truncation. For independent $N(0,\varepsilon^2)$ noise entries, truncating an oversampled intersection gives $O(\varepsilon^2)$ mean-square recovery error as $\varepsilon\to0$. If a row and a column remain unsampled, the corresponding error for final-product truncation is infinite at every $\varepsilon>0$ exactly when the row and column sample counts differ by at most one.

\end{abstract}



\section{Introduction}
\label{sec:introduction}
Low-rank approximation is a fundamental tool in numerical linear algebra, scientific computing, data analysis, and model reduction. Among low-rank approximation methods, CUR decompositions occupy a special role because they represent a matrix using actual columns and rows of the input matrix. Given \(A\in\mathbb R^{m\times n}\), a CUR approximation has the form \(A\approx C {X} R\), where \(C\) and \(R\) contain selected columns and rows of \(A\), and $X$ is a middle factor. In the standard skeleton form, if \(S\) and \(P\) select rows and columns, one writes \(A\approx AP(S^\top A P)^\dagger S^\top A\). Consequently, CUR decompositions are useful both as interpretable low-rank factorizations and as computational tools when only a subset of rows and columns can be accessed or stored.

The interpretability and sampling structure of CUR have led to applications in randomized numerical linear algebra, low-rank recovery, and reduced-order modeling \cite{mahoney2011randomized,mahoney2009cur,cai2021ircur,cai2021robustcur,donello2023cur,palkar2025adaptivecur,zheng2025semi}. Classical deterministic analyses establish exactness and error bounds based on conditioning and volume-type criteria \cite{goreinov1997theory,goreinov2001maximal,goreinov2010good}; randomized CUR has complementary relative-error guarantees \cite{drineas2008relative,halko2011finding}.

Perturbation analysis addresses a different but closely related question. Suppose that the observed matrix is \(A=M+E\), where \(M\) has rank \(r\) and \(E\) is a perturbation. If one forms a CUR approximation from \(A\), how accurately does it recover \(M\)? Existing perturbation results for CUR, including the finite-noise analysis in \cite{hammhuang2021perturbations}, provide bounds in unitarily invariant norms and show how the recovery error depends on the size of \(E\) and on the conditioning of the selected rows and columns. \rev{These bounds contain linear and quadratic terms that account separately for perturbations in the sampled rows, columns, and intersection. We identify the signed first-order contributions to the recovery error and show when they cancel.}

Let \(M=U\Sigma V^\top\in\mathbb R^{m\times n}\) have rank \(r\), and let \(S\) and \(P\) select \rev{\(s\)} rows and \rev{\(c\)} columns such that \(\rank(S^\top U)=r\) and \(\rank(V^\top P)=r\). \rev{With this admissible selection fixed, we truncate either the sampled intersection or the final CUR product, defining}
\[
    \Phi_r(A)=AP\rev{[(S^\top A P)_r]^\dagger}S^\top A,\qquad
    \rev{\Psi_r(A)=\bigl(AP(S^\top AP)^\dagger S^\top A\bigr)_r,}
\]
where $(\cdot)_r$ denotes rank-$r$ SVD truncation. Both maps recover every rank-$r$ matrix for which the samples are admissible, and they coincide for all inputs if $\min(s,c)=r$.

\rev{For a nonsingular $r\times r$ intersection ($s=c=r$), differentiating the ordinary skeleton formula gives the oblique tangent projector used by Dektor~\cite{dektor2025collocation} and Carrel et al.~\cite{carrel2025interpolatory} in dynamical low-rank approximation. Carrel and Grigori~\cite{carrel2026sketch} allow oversampling in their sketched projection of a vector field.}

\rev{For $s,c>r$, the clean intersection is rank deficient even when $s=c$. A small perturbation can introduce new nonzero singular values. The map $\Phi_r$ discards them before pseudoinversion, while $\Psi_r$ takes their reciprocals before truncating the product. For any admissible selection, Theorem~\ref{thm:cur-derivative} identifies the same projector as the Fr\'echet derivative of $\Phi_r$ at $M$.} Define the sampling-induced oblique projectors \(\Pi_U=U(S^\top U)^\dagger S^\top\) and \(\Pi_V=P(V^\top P)^\dagger V^\top\). Then
\[
    D\Phi_r(M)[E]=\Pi_U E+E\Pi_V-\Pi_U E\Pi_V\rev{\eqqcolon\mathcal I_{T_M}^{S,P}(E)}.
\]
Consequently, \(\Phi_r(M+E)-M=\mathcal I_{T_M}^{S,P}E+O(\|E\|^2)\). \rev{For rank-$r$ SVD truncation the corresponding derivative uses the orthogonal projectors $UU^\top$ and $VV^\top$ \cite{uschmajew2020geometric,lewis2008alternating}.}

\rev{For $\Phi_r$, Theorem~\ref{thm:mixed} shows that exact recovery implies $D\Phi_r(M)[E]=0$ for every $E$. The converse holds if $\|S^\top EP\|_2<\sigma_r(S^\top MP)$. For sufficiently small $\|E\|_2$, the Taylor remainder is bounded by $C\|E\|_2\|D\Phi_r(M)[E]\|$.}
\rev{For $\Psi_r$, Theorem~\ref{thm:post-truncation} gives a directional expansion with a quadratic remainder under its rank condition. Its first-order term includes both $\mathcal I_{T_M}^{S,P}(E)$ and a correction from perturbations in the null spaces of the clean intersection. Under its sampling and support hypotheses, Corollary~\ref{cor:post-instability} constructs unit directions $E$ for which $\Phi_r(M+tE)=M$ for all sufficiently small $t$, yet $\|\Psi_r'(M;E)\|$ can be arbitrarily large.}

\rev{For bounded, mean-zero isotropic noise $E$, the mean-square recovery error of $\Phi_r(M+\varepsilon E)$ is $\varepsilon^2\risk(S,P)+O(|\varepsilon|^3)$. The coefficient~\eqref{eq:risk} includes a product of row and column sampling terms. With fixed budgets and independently selectable admissible row and column sets, minimizing it reduces to two classical A-optimal design problems~\cite[Chapter~9]{pukelsheim2006optimal}. For $M+\varepsilon E$ with independent standard Gaussian entries in $E$ and $r<\min(s,c)$, Corollary~\ref{cor:core-gaussian} gives the same small-noise expansion for $\Phi_r$. If at least one row and one column remain unsampled, the mean-square recovery error of $\Psi_r$ at every $\varepsilon>0$ is infinite if $|s-c|\le1$ and finite if $|s-c|\ge2$ (Corollary~\ref{cor:post-gaussian}).}

\rev{Section~\ref{sec:preliminaries} introduces notation and preliminaries. Section~\ref{sec:main-theory} derives local CUR expansions and remainder bounds, characterizes exact recovery, and compares the recovery behavior of truncated SVD and the two rank-truncated CUR variants. Section~\ref{sec:risk-design} derives the leading mean-square error and sampling criteria for $\Phi_r$, quantifies oversampling gains, and compares the two truncation rules under Gaussian noise. Section~\ref{sec:numerics} presents numerical experiments, and Section~\ref{sec:conclusion} concludes. The appendices contain supporting proofs and error bounds for sampling scores computed from noisy matrix estimates.}

\section{Notation and preliminaries}
\label{sec:preliminaries}

We write \(\mathbb R^{m\times n}\) for the space of real \(m\times n\) matrices. For a matrix \(A\), we denote its rank, transpose, and Moore--Penrose pseudoinverse by \( \rank(A),  A^\top,  A^\dagger, \) respectively. \rev{We use \(\|A\|_2\) and \(\|A\|_F\) for the spectral and Frobenius norms, respectively, and \(\|A\|\) for a fixed normalized, uniformly generated unitarily invariant norm \cite{stewart1977perturbation}.} \rev{We write $DF(X)[H]$ for the Fr\'echet derivative of $F$ at $X$ applied to the direction $H$ \cite[Definition~3.1.5]{ortega2000iterative}.}

For the Moore--Penrose inverse restricted to the fixed-rank manifold, we write \( D(\cdot^\dagger)(W)[H] \) for the Fr\'echet derivative of the pseudoinverse map at \(W\) in the fixed-rank direction \(H\). When no confusion can arise, we abbreviate this as \( D(W^\dagger)[H]. \)

\rev{When a singular-value index exceeds $\min(s,c)$ for an $s\times c$ matrix, we set that singular value equal to zero.} \rev{This convention includes full-row-rank and full-column-rank boundary cases in the gap statements.}

For a positive integer \(r\), let \(A_r\) denote the rank-\(r\) \rev{SVD truncation} of \(A\).

\subsection{The rank-\texorpdfstring{\(r\)}{r} matrix manifold}

We recall standard facts about the fixed-rank matrix manifold \cite{vandereycken2013low}. Let \(\mathcal M_r = \{X\in\mathbb R^{m\times n}:\rank(X)=r\}.\)
This is a smooth embedded manifold of dimension \(r(m+n-r)\). 
Let \( M=U\Sigma V^\top\in\mathcal M_r \) be a compact SVD, where \( U\in\mathbb R^{m\times r}, V\in\mathbb R^{n\times r}, U^\top U=V^\top V=I_r, \) and \( \Sigma=\operatorname{diag}(\sigma_1,\ldots,\sigma_r), \sigma_1\ge\cdots\ge\sigma_r>0. \) The tangent space of \(\mathcal M_r\) at \(M\) is
\[
T_M\mathcal M_r = \{ UX^\top+YV^\top: X\in\mathbb R^{n\times r},\ Y\in\mathbb R^{m\times r} \}.
\]
Equivalently,
\(
T_M\mathcal M_r = \{ Z\in\mathbb R^{m\times n}: (I-UU^\top)Z(I-VV^\top)=0 \}.
\)
The \rev{Frobenius-}orthogonal projector onto \(T_M\mathcal M_r\) is
\[
\mathcal P_{T_M}(Z) = UU^\top Z+ZVV^\top-UU^\top ZVV^\top.
\]
The corresponding normal projector is \(\mathcal P_{T_M^\perp}(Z) = (I-UU^\top)Z(I-VV^\top).\)
Thus \( Z = \mathcal P_{T_M}Z + \mathcal P_{T_M^\perp}Z. \)

\subsection{Sampling matrices and rank-truncated CUR}
\label{subsec:sampling-cur}

Let \( S\in\mathbb R^{m\times s}, P\in\mathbb R^{n\times c} \) be column-selection matrices. Thus \(S^\top A\) consists of selected rows of \(A\), while \(AP\) consists of selected columns of \(A\). Equivalently, the columns of \(S\) and \(P\) are selected columns of the identity matrices \(I_m\) and \(I_n\), respectively. In particular, \( S^\top S=I_s, P^\top P=I_c, \) while \(SS^\top\) and \(PP^\top\) are coordinate projectors onto the selected row and column indices.

We write 
\(
C(A)=AP\),  \(R(A)=S^\top A\), \rev{and}   \(W(A)=S^\top A P. \)
The ordinary fixed-index CUR map is \( \Phi(A) = AP\,W(A)^\dagger S^\top A. \) In this paper we study the rank-truncated CUR map
\(\Phi_r(A) = AP\,W_r(A)^\dagger S^\top A,\)
where \( W_r(A)=(S^\top A P)_r. \) The rank truncation is important for local perturbation analysis. At an admissible rank-\(r\) matrix \(M\), the intersection \(W(M)\) has rank \(r\), but \(W(M+E)\) may have rank larger than \(r\) \rev{when $\min(s,c)>r$}. The ordinary pseudoinverse is not smooth across such rank changes, whereas the rank-truncated pseudoinverse remains on the fixed-rank manifold.

\subsection{Admissibility and exactness}
\label{subsec:admissibility}

\begin{definition} 
\label{def:admissible}
Let \(M=U\Sigma V^\top\in\mathcal M_r\). The sampling pair \((S,P)\) is called admissible for \(M\) if
\(\rank(S^\top U)=r, \qquad \rank(V^\top P)=r.\)
\end{definition}

Admissibility means that the selected columns capture the column space of \(M\), and the selected rows capture the row space of \(M\). In this case, \( W(M) = S^\top M P = (S^\top U)\Sigma(V^\top P) \) has rank \(r\).

We shall use the following standard exactness property of CUR \rev{\cite[Theorem~4.1]{hamm2020perspectives}}.

\begin{lemma}
\label{lem:exactness}
Let \(M=U\Sigma V^\top\in\mathcal M_r\), and suppose that \((S,P)\) is admissible for \(M\). Then \( \Phi_r(M)=M. \) Equivalently, \( M = MP(S^\top M P)^\dagger S^\top M. \)
\end{lemma}

\subsection{The sampling-induced tangent projector}
\label{subsec:interp-tangent-projector}

Admissible sampling induces two oblique projectors,
\[
\Pi_U = U(S^\top U)^\dagger S^\top, \qquad \Pi_V = P(V^\top P)^\dagger V^\top.
\]
They satisfy \( \Pi_U U=U, V^\top\Pi_V=V^\top. \) In general, \(\Pi_U\) and \(\Pi_V\) are not orthogonal projectors. Instead, \(\Pi_U\) is the oblique projector onto \(\operatorname{range}(U)\) determined by the selected row coordinates, and \(\Pi_V\) is the corresponding right-side oblique projector determined by the selected columns. \rev{$\Pi_V^\top$ is an oblique projector onto $\operatorname{range}(V)$.}

\begin{definition} 
\label{def:interp-tangent-projector}
Let \(M=U\Sigma V^\top\in\mathcal M_r\), and let \((S,P)\) be admissible for \(M\). The sampling-induced oblique tangent projector is the linear operator \(\mathcal I_{T_M}^{S,P}:\mathbb R^{m\times n}\to\mathbb R^{m\times n}\) defined by \(    \mathcal I_{T_M}^{S,P}(E) = \Pi_U E+E\Pi_V-\Pi_U E\Pi_V.
\)
\end{definition}

This operator has the same inclusion--exclusion form as the orthogonal tangent projector \(\mathcal P_{T_M}\), but with \(UU^\top\) and \(VV^\top\) replaced by the sampling-induced oblique projectors \(\Pi_U\) and \(\Pi_V\) \rev{\cite{dektor2025collocation,carrel2026sketch}}.

\begin{lemma} 
\label{lem:interp-tangent-projector}
Let \(M=U\Sigma V^\top\in\mathcal M_r\), and let \((S,P)\) be admissible for \(M\). Then \( \mathcal I_{T_M}^{S,P} \) is a projector onto \(T_M\mathcal M_r\). More precisely,
\(
\range\bigl(\mathcal I_{T_M}^{S,P}\bigr) \subseteq T_M\mathcal M_r,
\)
and
\(\mathcal I_{T_M}^{S,P}(E)=E \qquad \text{for all } E\in T_M\mathcal M_r.\)
Consequently, \( \bigl(\mathcal I_{T_M}^{S,P}\bigr)^2 = \mathcal I_{T_M}^{S,P}. \)
\end{lemma}

\begin{proof}
\rev{The column spaces of $\Pi_UE$ and $\Pi_UE\Pi_V$ lie in $\operatorname{range}(U)$, while the row space of $E\Pi_V$ lies in the row space of $V^\top$.} \rev{Hence $\mathcal I_{T_M}^{S,P}(E)$ belongs to $T_M\mathcal M_r$.} \rev{If $E=UX^\top+YV^\top$ is tangent, the identities $\Pi_UU=U$ and $V^\top\Pi_V=V^\top$ give} \[
\begin{aligned}
     \mathcal I_{T_M}^{S,P}(E)=&\Pi_U(UX^\top+YV^\top)+(UX^\top+YV^\top)\Pi_V-\Pi_U(UX^\top+YV^\top)\Pi_V\\
     =&UX^\top+\Pi_UYV^\top+UX^\top\Pi_V+YV^\top-UX^\top \Pi_V-\Pi_U YV^\top\\
     =&UX^\top+YV^\top=E.
\end{aligned}
\]
  \rev{The operator is therefore a projector onto this tangent space.}
\end{proof}

\section{Local expansion of rank-truncated CUR}
\label{sec:main-theory}

\subsection{A supporting rank-truncation expansion}

The following lemma is a quantitative version of the standard local expansion of the truncated SVD as the \rev{Frobenius} metric projection onto the fixed-rank manifold. The derivative of this projection is the orthogonal tangent projector; see \cite{uschmajew2020geometric,lewis2008alternating}. \rev{Vu, Chunikhina, and Raich \cite[Corollary~2 and Theorem~3]{vu2021perturbation} give the first-order expansion and an explicit Frobenius-norm bound for its remainder.} \rev{The mixed-norm estimate below retains explicit dependence on the singular-value margin.} \rev{Its proof is given in Appendix~\ref{app:rank-truncation-proof}.}

\begin{lemma}
\label{lem:rank-r-truncation-first-order}
Let \(W\in\mathbb R^{s\times c}\) have rank \(r\), and let
\(W=Q\Lambda Z^\top\) be a compact SVD, where
\(Q^\top Q=Z^\top Z=I_r,\quad
\Lambda=\operatorname{diag}(\lambda_1,\ldots,\lambda_r),\quad
\gamma:=\sigma_r(W)=\lambda_r>0.\)
Let \(\widehat W=(W+\rev{\Delta})_r\). Assume \(\|\rev{\Delta}\|_2\le\rev{\theta}\gamma\)
for some fixed \(0<\rev{\theta}<1/2\). Then
\rev{\(\widehat W\) is uniquely determined and has rank \(r\), and}
\(\widehat W=W+\mathcal P_{T_W}\rev{\Delta}+R_{\mathrm{tr}}(\rev{\Delta}),\)
where
\(\mathcal P_{T_W}\rev{\Delta}
=QQ^\top\rev{\Delta}+\rev{\Delta}ZZ^\top-QQ^\top\rev{\Delta}ZZ^\top\)
\rev{and \(R_{\mathrm{tr}}(\Delta):=(W+\Delta)_r-W-\mathcal P_{T_W}\Delta\)}.
Moreover,
\begin{equation}\label{eq:tsvd-spectral-remainder}
\rev{\|R_{\mathrm{tr}}(\Delta)\|\le\frac{3\|\Delta\|_2\|\Delta\|}{\gamma-\|\Delta\|_2}\le\frac{3}{1-\theta}\frac{\|\Delta\|_2\|\Delta\|}{\gamma}.}
\end{equation}
\end{lemma}

\subsection{Derivative of the rank-truncated CUR map}

We now apply Lemma~\ref{lem:rank-r-truncation-first-order} to the intersection matrix \( W(A)=S^\top A P. \) At the base point \(M\), set \( W=S^\top M P. \) By admissibility, \( W=(S^\top U)\Sigma(V^\top P) \) has rank \(r\).

\begin{theorem} 
\label{thm:cur-derivative}
Let \( M=U\Sigma V^\top\in\mathbb R^{m\times n} \) have rank \(r\), and suppose that \((S,P)\) is admissible for \(M\). Define \( \Phi_r(A) = AP\rev{[(S^\top A P)_r]^\dagger} S^\top A. \) Then \(\Phi_r\) is \rev{$C^\infty$ on an open neighborhood of $M$}, and for every perturbation \(E\in\mathbb R^{m\times n}\),
\[
D\Phi_r(M)[E] = \Pi_U E+E\Pi_V-\Pi_U E\Pi_V = \mathcal I_{T_M}^{S,P}(E).
\]
Consequently, for all sufficiently small \(E\),
\[
\Phi_r(M+E) = M+\mathcal I_{T_M}^{S,P}(E)+O(\|E\|^2),
\]
where the remainder is measured in \rev{the chosen unitarily invariant} norm. \rev{For fixed \(\theta\in(0,1/2)\)}, the rank-truncation step is controlled whenever \( \|S^\top E P\|_2 \le {\rev{\theta}}\,\sigma_r(S^\top M P). \) \rev{The remainder constant may depend on \(M,S,P\), the chosen norm, and \(\theta\).}
\end{theorem}

\begin{proof}
Let \( A(t)=M+tE. \) Then \(    W(t):=S^\top A(t)P = W+t\Delta, \Delta=S^\top E P. \) By Weyl's inequality,
\[
    \sigma_r(W+t\Delta)
    \ge
    \sigma_r(W)-|t|\|\Delta\|_2,
    \qquad
    \sigma_{r+1}(W+t\Delta)
    \le
    |t|\|\Delta\|_2.
\]
Thus, for \(|t|\|\Delta\|_2<\sigma_r(W)/2\), \rev{the gap is positive.} \rev{Near \(W\), rank-\(r\) SVD truncation is the nearest-point projection onto the smooth fixed-rank manifold, and its smoothness follows from \cite[Lemma~2.1]{lewis2008alternating}.}
Hence \(W_r(t)=(W+t\Delta)_r\) is a smooth fixed-rank curve through \(W\). By Lemma~\ref{lem:rank-r-truncation-first-order}, \(W_r(t)=W+t\rev{\mathcal P_{T_W}\Delta}+O(t^2)\).
Set \(H:=\dot W_r(0)=\mathcal P_{T_W}\Delta\).
\rev{The preceding singular-value bound gives \(\|W_r(t)^\dagger\|_2\le 2/\sigma_r(W)\), so we may take \(\alpha=2\). By Lemma~\ref{lem:rank-r-truncation-first-order}, \(\|W_r(t)-W\|_2\|W_r(t)-W\|=O(t^2)\) and \(\|W_r(t)-W-tH\|=O(t^2)\). Applying Lemma~\ref{lem:quantitative-pinv-remainder} with \(X=W\), \(Y=W_r(t)\), and tangent increment \(tH\) therefore yields} \(W_r(t)^\dagger=W^\dagger+tD(W^\dagger)[H]+O(t^2)\).

Expanding the CUR map gives
\[
\begin{aligned}
    \Phi_r(A(t))
    &=(MP+tEP)\left(W^\dagger+tD(W^\dagger)[H]+O(t^2)\right)
                      (S^\top M+tS^\top E)\\
    &=MPW^\dagger S^\top M
      +t\Big[EPW^\dagger S^\top M+MPD(W^\dagger)[H]S^\top M+MPW^\dagger S^\top E\Big]+O(t^2).
\end{aligned}
\]
By Lemma~\ref{lem:exactness}, \( MPW^\dagger S^\top M=M. \)

It remains to simplify the first-order coefficient. \rev{Since \(H\in T_W\mathcal M_r\),} Lemma~\ref{lem:cur-pinv-consequence} \rev{gives} \( MPD(W^\dagger)[H]S^\top M = -MPW^\dagger H W^\dagger S^\top M. \)

Now \( H=\mathcal P_{T_W}\Delta = WW^\dagger\Delta + \Delta W^\dagger W - WW^\dagger\Delta W^\dagger W. \) Using \(W^\dagger W W^\dagger=W^\dagger\), we obtain \( W^\dagger H W^\dagger = W^\dagger\Delta W^\dagger. \) Thus, \(MPD(W^\dagger)[H]S^\top M=-MPW^\dagger\Delta W^\dagger S^\top M.\)
Since \(\Delta=S^\top E P\),  \( MPD(W^\dagger)[H]S^\top M = -MPW^\dagger S^\top E P W^\dagger S^\top M. \) Hence
\[
\begin{aligned}
    D\Phi_r(M)[E]
    &=
    EPW^\dagger S^\top M
    +
    MPW^\dagger S^\top E  
    -
    MPW^\dagger S^\top E P W^\dagger S^\top M.
\end{aligned}
\]
It remains to identify the two oblique projectors. Put \(A_S=S^\top U, B_P=V^\top P.\)
Then \( W=A_S\Sigma B_P. \) Since \(A_S\) has full column rank and \(B_P\) has full row rank, \( W^\dagger=B_P^\dagger\Sigma^{-1}A_S^\dagger. \) Therefore,
\[
\begin{aligned}
    MPW^\dagger S^\top
    &=
    U\Sigma B_P
    B_P^\dagger\Sigma^{-1}A_S^\dagger S^\top  =
    U(S^\top U)^\dagger S^\top
    =
    \Pi_U.
\end{aligned}
\]
Similarly, \(PW^\dagger S^\top M=P B_P^\dagger\Sigma^{-1}A_S^\dagger A_S\Sigma V^\top
=P(V^\top P)^\dagger V^\top=\Pi_V\). \rev{Hence}
\[
    D\Phi_r(M)[E]
    =
    \Pi_U E+E\Pi_V-\Pi_U E\Pi_V
    =
    \mathcal I_{T_M}^{S,P}(E).
\]
\rev{Near \(M\), the truncated core \((S^\top AP)_r\) is smooth of rank \(r\). By smooth fixed-rank pseudoinversion~\cite[Lemma~1]{guerra2020necessity}, \(\Phi_r\) is \(C^\infty\), with the Fr\'echet derivative above and a uniform quadratic Taylor remainder.}
\end{proof}

Write \(\mathcal R(E)=\Phi_r(M+E)-M-\mathcal I_{T_M}^{S,P}(E)\) for the Taylor remainder. In particular, there are $\rho>0$ and $C_2<\infty$ such that
\begin{equation}\label{eq:uniform-taylor}\|\mathcal R(E)\|\le C_2\|E\|_2\|E\|\quad\text{for }\|E\|_2\le\rho.
\end{equation} Corollary~\ref{cor:quantitative-cur-expansion} gives an explicit bound in terms of the sampled perturbations.

\begin{corollary}
\label{cor:quantitative-cur-expansion}
\rev{Under the assumptions of Theorem~\ref{thm:cur-derivative}, set \(\gamma=\sigma_r(S^\top MP)\), \(a=\|MP\|_2\), \(b=\|S^\top M\|_2\), and \(\delta=\|S^\top EP\|_2\).} \rev{If \(\delta\le\theta\gamma\) for \(\theta\in(0,1/2)\), then \(\mathcal R(E)\) satisfies}
 \begin{equation}\label{eq:structured-CUR-remainder}
\begin{aligned}
\|\mathcal R(E)\|\le{}&c_1(\theta)\frac{ab\delta\|S^\top EP\|}{\gamma^3}
+c_2(\theta)\frac{\delta\bigl(b\|EP\|+a\|S^\top E\|\bigr)}{\gamma^2}\\
&+\frac{c_3(\theta)}{\gamma}\min\{\|EP\|\|S^\top E\|_2,\|EP\|_2\|S^\top E\|\},
\end{aligned}
\end{equation} 
\rev{where the positive coefficients \(c_j(\theta)\) are given explicitly in \eqref{eq:cur-remainder-coefficients}.}
\end{corollary}

\begin{remark}\label{rem:remainder-use}
\rev{Let \(B_\theta(E)\) denote the right-hand side of \eqref{eq:structured-CUR-remainder}.} Then
 \begin{equation}\label{eq:recovery-bracket}
\left|\|\Phi_r(M+E)-M\|-\|\mathcal I_{T_M}^{S,P}(E)\|\right|\le B_\theta(E).
\end{equation} 
\rev{For a fixed direction with \(d(E):=\|\mathcal I_{T_M}^{S,P}(E)\|>0\), the relative error of the linear prediction at \(tE\) is at most \(|t|B_\theta(E)/d(E)\), provided \(|t|\|S^\top EP\|_2\le\theta\gamma\).} \rev{Appendix~\ref{app:explicit-cur-proof} gives the coefficient formulas, the proof, and a uniform bound in terms of \(\|E\|_2\|E\|\).}
\end{remark}

\subsection{\rev{Exact recovery and relative linearization}}

\rev{If \(S^\top E=0\) and \(EP=0\), every sampled input remains unchanged, and admissible CUR gives \(\Phi_r(M+tE)=M\) for all \(t\in\mathbb R\).}

\begin{theorem}\label{thm:mixed}
\rev{Under the assumptions of Theorem~\ref{thm:cur-derivative}, set \(W=S^\top MP\).}
\begin{enumerate}[label=\textup{(\roman*)},leftmargin=*]
\item\label{item:kernel-exactness} \rev{For every \(E\), \(\Phi_r(M+E)=M\) implies \(\mathcal I_{T_M}^{S,P}(E)=0\).}
\rev{If \(\|S^\top EP\|_2<\sigma_r(W)\), the converse also holds, giving}
\begin{equation}\label{eq:local-exact-recovery}
\rev{\Phi_r(M+E)=M\quad\Longleftrightarrow\quad\mathcal I_{T_M}^{S,P}(E)=0.}
\end{equation}
\item \rev{For the chosen unitarily invariant norm, there exist \(\rho_0>0\) and \(C_{\mathrm{rel}}>0\), depending on \(M,S,P\) and the norm, such that}
 \begin{equation}\label{eq:mixed}
\|\mathcal R(E)\| \le C_{\mathrm{rel}}\|E\|_2\|\mathcal I_{T_M}^{S,P}(E)\|,\qquad \|E\|_2<\rho_0.
\end{equation} 
\end{enumerate}
\end{theorem}

\begin{proof}
\rev{For \textup{(i)}, write \(W=Q\Lambda Z^\top\), \(\Delta=S^\top EP\), and \(D=W+\Delta\). By idempotence, \(\mathcal I_{T_M}^{S,P}(E)=0\) if and only if \(\Pi_UE=E\Pi_V=0\), or equivalently \(Q^\top S^\top E=0\) and \(EPZ=0\), since \(\range(Q)=\range(S^\top U)\) and \(\range(Z)=\range(P^\top V)\).}

\rev{Suppose \(\Phi_r(M+E)=M\). Sampling its intersection gives \(W=D(D_r)^\dagger D=D_r\). Hence \(DW^\dagger=QQ^\top\) and \(W^\dagger D=ZZ^\top\) by the SVD of \(D\). Sampling the rows and columns of the recovery identity now gives}
 \[
\begin{aligned}
S^\top M
 &= S^\top\!\bigl[(M+E)PD_r^\dagger S^\top(M+E)\bigr]
  = DD_r^\dagger S^\top(M+E) = QQ^\top S^\top(M+E),\\
MP &= (M+E)PD_r^\dagger D = (M+E)PZZ^\top.
\end{aligned}
\] 
\rev{Since \(QQ^\top S^\top M=S^\top M\) and \(MPZZ^\top=MP\), we obtain \(Q^\top S^\top E=0\) and \(EPZ=0\), proving necessity.}

\rev{Conversely, if \(\mathcal I_{T_M}^{S,P}(E)=0\), then \(Q^\top\Delta=0\) and \(\Delta Z=0\). Thus}
 \[
[Q,Q_\perp]^\top(W+\Delta)[Z,Z_\perp]=\begin{bmatrix}\Lambda&0\\0&\Delta_{22}\end{bmatrix},\qquad \Delta_{22}=Q_\perp^\top\Delta Z_\perp.
\] 
\rev{If \(\|\Delta\|_2<\sigma_r(W)\), then \(\|\Delta_{22}\|_2<\sigma_r(W)\) and \(D_r=W\). Since \(W^\dagger=Z\Lambda^{-1}Q^\top\), we have \(EPW^\dagger=0\) and \(W^\dagger S^\top E=0\). Therefore,}
 \[
\Phi_r(M+E)=(MP+EP)W^\dagger(S^\top M+S^\top E)=MPW^\dagger S^\top M=M.
\] 

\rev{For \textup{(ii)}, put \(H=\mathcal I_{T_M}^{S,P}(E)\). Since \(\mathcal I_{T_M}^{S,P}\) is a projector by Lemma~\ref{lem:interp-tangent-projector}, \(\mathcal I_{T_M}^{S,P}(E-H)=0\) and \(D\Phi_r(M)[H]=H\). For \(0\le t\le1\),}
 \[
\|E-(1-t)H\|_2\le\|E\|_2+\|H\|_2\le\bigl(1+\|\mathcal I_{T_M}^{S,P}\|_{2\to2}\bigr)\|E\|_2.
\] 
\rev{Take \(E\) sufficiently small so that \(\|E-H\|_2<\sigma_r(W)\) and the whole segment lies in a smooth neighborhood of \(M\). Since \(\|S^\top(E-H)P\|_2\le\|E-H\|_2\), part \textup{(i)} gives \(\Phi_r(M+E-H)=M\). The fundamental theorem of calculus yields}
 \[
\mathcal R(E)=\int_0^1\bigl[D\Phi_r(M+E-(1-t)H)-D\Phi_r(M)\bigr][H]\,dt.
\] 
\rev{For \(X\) in a sufficiently small spectral-norm ball,}
 \[
\bigl[D\Phi_r(M+X)-D\Phi_r(M)\bigr][H]=\int_0^1 D^2\Phi_r(M+uX)[X,H]\,du.
\] 
\rev{By continuity of \(D^2\Phi_r\), the bilinear bound \(\|D^2\Phi_r(Y)[X,H]\|\le L\|X\|_2\|H\|\) holds uniformly on a small closed ball about \(M\) for some \(L>0\). Integrating with \(X=E-(1-t)H\) gives~\eqref{eq:mixed} with \(C_{\mathrm{rel}}=L(1+\|\mathcal I_{T_M}^{S,P}\|_{2\to2})\).}
\end{proof}

\subsection{Comparison with SVD truncation}

The corresponding first-order expansion for the rank-\(r\) SVD truncation is classical. Since the truncated SVD is the local \rev{Frobenius} metric projection onto the rank-\(r\) manifold, one has \( (M+E)_r-M = \mathcal P_{T_M}E+O(\|E\|^2), \) where \(\mathcal P_{T_M}E = UU^\top E+EVV^\top-UU^\top EVV^\top.\)
This follows, for example, from the metric-projection interpretation of the truncated SVD \cite{uschmajew2020geometric} and the general smooth projection result \cite[Lemma~2.1]{lewis2008alternating}.

Thus SVD and CUR retain different first-order components of the perturbation:
\(\text{SVD retains } \mathcal P_{T_M}E, \qquad \text{CUR retains } \mathcal I_{T_M}^{S,P}E.\)
\rev{If $\mathcal P_{T_M}E=0$, then $(M+tE)_r=M$ whenever $|t|\|E\|_2<\sigma_r(M)$, because the signal and perturbation occupy orthogonal diagonal blocks in the singular-vector bases of $M$.} \rev{The analogous CUR statement is Theorem~\ref{thm:mixed}\textup{(i)}.} Neither method is uniformly better. Their local behavior depends on the geometry of the perturbation.

\subsection{\rev{Relation to finite-noise perturbation bounds}}

\rev{Hamm and Huang \cite{hammhuang2021perturbations} study the same reconstruction with a rank-truncated sampled intersection among several CUR variants and derive finite-noise bounds with separate sampled-row, sampled-column, and intersection contributions.} \rev{Their estimates already include terms of first and second order in the noise.} \rev{The present expansion identifies the exact signed leading term and bounds the remainder after subtracting it.}

\rev{Write $E_I=S^\top E$, $E_J=EP$, and $E_{I,J}=S^\top EP$.} Then
\begin{equation}\label{eq:signed-sampled-error}
\rev{D\Phi_r(M)[E]=U(S^\top U)^\dagger E_I+E_J(V^\top P)^\dagger V^\top-U(S^\top U)^\dagger E_{I,J}(V^\top P)^\dagger V^\top.}
\end{equation}
\rev{With \(\alpha_U=\|(S^\top U)^\dagger\|_2\) and \(\alpha_V=\|(V^\top P)^\dagger\|_2\), the product inequality for \(\|\cdot\|\) gives}
 \begin{equation}\label{eq:local-row-column-bound}
\|D\Phi_r(M)[E]\|\le\alpha_U\|E_I\|+\alpha_V\|E_J\|+\alpha_U\alpha_V\|E_{I,J}\|.
\end{equation}

\rev{In the spectral norm, the first-order intersection term in the upper bound of \cite[Theorem~4.6]{hammhuang2021perturbations} is \(4\alpha_U\alpha_V\|E_{I,J}\|_2\). Our derivative bound~\eqref{eq:local-row-column-bound} has the same term with coefficient \(1\). The signed derivative~\eqref{eq:signed-sampled-error} preserves cancellation before norms are taken.} \rev{Corollary~\ref{cor:quantitative-cur-expansion} explicitly controls the difference between this signed term and the nonlinear recovery error.}

\subsection{\rev{Local analysis of rank-truncated CUR variants}}
\label{sec:post-truncation}
 
The position of the rank truncation changes the local behavior. Whereas
\(\Phi_r\) truncates the intersection before pseudoinversion, consider
\(\Psi_r(A)=\bigl(AP(S^\top AP)^\dagger S^\top A\bigr)_r.\)
If \(\min(s,c)=r\), then \(\Psi_r=\Phi_r\); below we assume \(s,c>r\).
Numerical comparisons of these variants appear in
\cite[Section~10]{hammhuang2021perturbations}. Here we examine their behavior
at an exact rank-\(r\) signal. Write \(W=S^\top MP=Q\Lambda Z^\top\)
as a compact SVD and choose orthonormal complements \(Q_\perp,Z_\perp\).
For a perturbation \(E\), set \(\Delta=S^\top EP\) and
\(\Delta_{22}(E)=Q_\perp^\top\Delta Z_\perp\), and define
\begin{equation}\label{eq:post-noise-term}
 K(E)=(I-\Pi_U)EPZ_\perp\Delta_{22}(E)^\dagger Q_\perp^\top S^\top E(I-\Pi_V).
\end{equation}

\begin{theorem}
\label{thm:post-truncation}
Let \(M\) have rank \(r\), and let \((S,P)\) be admissible for \(M\), with \(s,c>r\).
Set \(k=\rank\Delta_{22}(E)\). Suppose that, for some \(t_0>0\),
\[
 \rank\bigl(S^\top(M+tE)P\bigr)=r+k\qquad\text{for }0<|t|<t_0.
\]
The curves \(t\mapsto (M+tE)P(S^\top(M+tE)P)^\dagger S^\top(M+tE)\)
and \(t\mapsto\Psi_r(M+tE)\) are smooth through \(t=0\). Their derivatives
at zero are, respectively,
\begin{equation}\label{eq:post-directional}
 \mathcal I_{T_M}^{S,P}(E)+K(E),\qquad
 \Psi_r'(M;E)=\mathcal I_{T_M}^{S,P}(E)+\mathcal P_{T_M}K(E).
\end{equation}
The corresponding Taylor remainders are \(O(t^2)\) in \(\|\cdot\|\) for each fixed \(E\).
Moreover, \(K(\lambda E)=\lambda K(E)\) for \(\lambda\ne0\).
\end{theorem}

If \(E\) has independent standard Gaussian entries, then
\(\Delta_{22}(E)\) has rank \(\min(s,c)-r\) almost surely.
Consequently, the theorem's rank hypothesis holds almost surely for all sufficiently small nonzero \(t\).
Corollary~\ref{cor:post-gaussian} in Section~\ref{sec:risk-design} determines when the directional derivative of \(\Psi_r\) has finite second moment.
The extra term in~\eqref{eq:post-directional} involves inversion of noise
in the null spaces of the clean intersection. It is generally nonlinear
in \(E\) and can be arbitrarily large even along exact-recovery directions
of \(\Phi_r\), as the following consequence shows.

\begin{corollary}
\label{cor:post-instability}
Let \(M\) have rank \(r\), and let \((S,P)\) be admissible, with
\(r<s<m\), \(r<c<n\), and at least one of
\((I-SS^\top)U\) and \((I-PP^\top)V\) nonzero.
For every integer \(1\le k\le\min(s,c)-r\) and every \(L>0\), there is a
direction \(E\in\ker\mathcal I_{T_M}^{S,P}\) with \(\|E\|=1\) and
\(\rank\Delta_{22}(E)=k\) such that, for sufficiently small nonzero \(t\),
\[
 \rank\bigl(S^\top(M+tE)P\bigr)=r+k,\qquad \Phi_r(M+tE)=M,
\]
whereas
\[
 \lim_{t\to0}\frac{\|\Psi_r(M+tE)-M\|}{\|tE\|}>L.
\]
Consequently, \(\Psi_r\) has no uniform local linear error bound at \(M\)
and is not Fr\'echet differentiable there.
\end{corollary}

The proofs are in Appendix~\ref{app:post-truncation}.
\(\Phi_r\) is smooth on an open neighborhood of \(M\), whereas
Theorem~\ref{thm:post-truncation} gives smoothness of \(t\mapsto\Psi_r(M+tE)\)
for each fixed \(E\) satisfying its rank condition. The following example
shows that, when the direction varies, the recovery error of \(\Psi_r\) can
diverge as the perturbation vanishes, while \(\Phi_r\) remains exact.

\begin{example}
\label{ex:post-discontinuity}
Let \(r=1\), select the first three rows and columns in \(\R^{4\times4}\),
and, for \(0<\eta\le1\), set
\[
 M=\tfrac12(e_1+e_4)(e_1+e_4)^\top,\qquad
 E_\eta=\eta e_2e_2^\top+e_2e_4^\top+e_4e_2^\top.
\]
Here \(\|M\|=1\). For \(0<t\eta<1/2\), direct multiplication gives
\(\Phi_1(M+tE_\eta)=M\) and
\[
 \Psi_1(M+tE_\eta)
 =\left[M+\frac{t}{\eta}(\eta e_2+e_4)(\eta e_2+e_4)^\top\right]_1.
\]
Take \(\eta=t^2\) and let \(t\downarrow0\). Then
\(\|tE_{t^2}\|\le2t+t^3\to0\), whereas
\[
 \|\Psi_1(M+tE_{t^2})-M\|\ge t^{-1}-\tfrac12\longrightarrow\infty.
\]
Indeed, the matrix inside the truncation is positive semidefinite and its
\(e_4\)-Rayleigh quotient is \(t^{-1}+1/2\); the norm of its rank-one truncation equals
its largest eigenvalue.
\end{example}


\section{\rev{Mean-square recovery and sampling design}}

\label{sec:comparison}\label{sec:risk-design}
\rev{Throughout this section, let $M$ have rank $1\le r<\min(m,n)$ and fix admissible row and column indices independently of the recovery noise.} \rev{Write $G_S=U^\top SS^\top U$, $H_P=V^\top PP^\top V$, and $a_S=\tr(G_S^{-1})$, $b_P=\tr(H_P^{-1})$.} \rev{These Gram matrices are positive definite, with $a_S,b_P\ge r$.}

\begin{theorem}\label{thm:risk}
\rev{Let $Z\in\R^{m\times n}$ be centered, satisfy $\E[\operatorname{vec}(Z)\operatorname{vec}(Z)^\top]=I_{mn}$, and obey $\norm{Z}_F\le B$ almost surely.} \rev{Then, as $\varepsilon\to0$,}
 \begin{equation}\label{eq:risk-limit}
\E\norm{\Phi_r(M+\varepsilon Z)-M}_F^2=\varepsilon^2\risk(S,P)+O(|\varepsilon|^3),
\end{equation} 
\rev{where}
 \begin{equation}\label{eq:risk}
\risk(S,P)=r(m+n-r)+(n-r)(a_S-r)+(m-r)(b_P-r)+(a_S-r)(b_P-r).
\end{equation} 
\end{theorem}

\begin{proof}
\rev{Extend $U,V$ to orthonormal bases $\mathcal U=[U,U_\perp]$ and $\mathcal V=[V,V_\perp]$.} \rev{In these coordinates, the sampling projectors have the block forms $\mathcal U^\top\Pi_U\mathcal U=\left[\begin{smallmatrix}I&K_U\\0&0\end{smallmatrix}\right]$ and $\mathcal V^\top\Pi_V\mathcal V=\left[\begin{smallmatrix}I&0\\K_V&0\end{smallmatrix}\right]$, where $\norm{K_U}_F^2=a_S-r$ and $\norm{K_V}_F^2=b_P-r$.} \rev{If $\mathcal U^\top Z\mathcal V=[Z_{ij}]_{i,j=1}^2$, the derivative theorem yields}
 \begin{equation}\label{eq:block-action}
\mathcal U^\top\mathcal I_{T_M}^{S,P}(Z)\mathcal V=\begin{bmatrix}Z_{11}-K_UZ_{22}K_V&Z_{12}+K_UZ_{22}\\Z_{21}+Z_{22}K_V&0\end{bmatrix}.
\end{equation} 
\rev{Identity covariance is preserved by these orthogonal changes of coordinates. For fixed compatible matrices $A_1,A_2$ and a centered random matrix $G$ with identity covariance, $\E[G^\top A_1^\top A_1G]=\|A_1\|_F^2I$, and hence}
 \begin{equation}\label{eq:matrix-second-moment}
\E\|A_1GA_2\|_F^2=\tr\!\left(A_2^\top\E[G^\top A_1^\top A_1G]A_2\right)=\|A_1\|_F^2\|A_2\|_F^2.
\end{equation} 
\rev{In \eqref{eq:matrix-second-moment}, take $G=Z_{22}$ and choose $(A_1,A_2)$ successively as
$(K_U,I_{n-r})$, $(I_{m-r},K_V)$, and $(K_U,K_V)$. Since $\|I_d\|_F^2=d$,
the corresponding values of $\E\|A_1Z_{22}A_2\|_F^2$ are}
 \[
 (n-r)\|K_U\|_F^2,\qquad (m-r)\|K_V\|_F^2,\qquad \|K_U\|_F^2\|K_V\|_F^2.
\] 
\rev{The terms $Z_{11},Z_{12},Z_{21}$ contribute $r^2+r(n-r)+(m-r)r=r(m+n-r)$. Distinct transformed blocks are uncorrelated, so the cross terms in \eqref{eq:block-action} have zero expectation. Substituting $\|K_U\|_F^2=a_S-r$ and $\|K_V\|_F^2=b_P-r$ proves \eqref{eq:risk}.} \rev{Finally, specializing~\eqref{eq:uniform-taylor} to the Frobenius norm and using $\norm{Z}_2\le\norm{Z}_F$ gives a uniform $O(\varepsilon^2\norm{Z}_F^2)$ remainder for bounded $Z$.} \rev{Squaring and taking expectations proves \eqref{eq:risk-limit}.} \rev{An explicit finite-$\varepsilon$ estimate is given in Appendix~\ref{app:risk-quantitative}.}
\end{proof}

\rev{In \eqref{eq:risk}, $r(m+n-r)$ is the leading mean-square error coefficient for rank-$r$ SVD truncation under isotropic noise. The normal block $Z_{22}$ in \eqref{eq:block-action} contributes the three sampling terms, including the row--column interaction $(a_S-r)(b_P-r)$. Carrel and Grigori \cite[Section~4]{carrel2026sketch} analyze how the normal component of the vector field affects the projected dynamics under a sketched Galerkin condition. These terms depend only on the sampled subspaces. The signal singular values affect the remainder bound in Corollary~\ref{cor:quantitative-cur-expansion}.}

\rev{For centered noise with identity covariance, \eqref{eq:risk} gives the linearized mean-square error, and boundedness controls the remainder in \eqref{eq:risk-limit}.}

\begin{corollary}\label{cor:aopt}
\rev{Fix row and column sampling budgets and consider a Cartesian product of admissible row and column families.} \rev{A pair $(S_*,P_*)$ minimizes $\risk(S,P)$ if and only if $S_*$ minimizes $\tr(G_S^{-1})$ over the row family and $P_*$ minimizes $\tr(H_P^{-1})$ over the column family.}
\end{corollary}

\begin{proof}
\rev{The right-hand side of \eqref{eq:risk}, regarded as a function of $a_S,b_P$, is strictly increasing in each variable: its partial derivatives are $n+b_P-2r\ge n-r>0$ and $m+a_S-2r\ge m-r>0$.} \rev{The product structure of the admissible families gives the conclusion.}
\end{proof}

\rev{Corollary~\ref{cor:aopt} reduces the coupled objective in \eqref{eq:risk} to two classical A-optimal design problems \cite{pukelsheim2006optimal}, one for the rows and one for the columns. For approximation algorithms for discrete A-optimal design, see \cite{nikolov2019aoptimal}. We next quantify the reduction in the leading mean-square error~\eqref{eq:risk} when one row or column is added to the current selection.}

\begin{proposition}\label{prop:gain}
\rev{Adding sampled rows or columns cannot increase $\risk(S,P)$.} \rev{More precisely, for an unselected row $i$ let $u_i=U^\top e_i$ and let $S_+$ add that row. Then}
 \begin{equation}\label{eq:row-gain}
\risk(S,P)-\risk(S_+,P)=(n+b_P-2r)\frac{u_i^\top G_S^{-2}u_i}{1+u_i^\top G_S^{-1}u_i}.
\end{equation} 
\rev{For an unselected column $j$, with $v_j=V^\top e_j$, the corresponding decrease is $(m+a_S-2r)v_j^\top H_P^{-2}v_j/(1+v_j^\top H_P^{-1}v_j)$.}
\end{proposition}

\begin{proof}
\rev{The matrix $G_{S_+}=G_S+u_iu_i^\top$ has inverse $G_S^{-1}-G_S^{-1}u_iu_i^\top G_S^{-1}/(1+u_i^\top G_S^{-1}u_i)$ by the Sherman--Morrison formula~\cite{golubvanloan2013matrix}.} \rev{Taking traces and using the coefficient $n+b_P-2r$ of $a_S$ in \eqref{eq:risk} proves \eqref{eq:row-gain}; the column case is identical.} \rev{Monotonicity follows from the nonnegativity of these decreases.}
\end{proof}

\rev{A greedy step ranks candidate rows by \eqref{eq:row-gain} using the $r\times r$ Gram matrix $G_S$. Its optimality is limited to the next row addition.}

\rev{The leading left and right singular subspaces of an existing estimate $M+E$ give row and column scores for fresh observations of $M$ with independent isotropic noise. Proposition~\ref{prop:pilot} in Appendix~\ref{app:additional-estimates} bounds the resulting score errors. Section~\ref{sec:pilot-decisions} evaluates the selected indices conditional on the estimate, whose acquisition cost is separate.}

\rev{For bounded centered noise with covariance $\Gamma$, the leading coefficient is $\tr(J\Gamma J^\top)$, where $J$ represents $\mathcal I_{T_M}^{S,P}$ in orthonormal vectorized coordinates. The separability and monotonicity results above assume isotropy.}

We now compare the two truncation rules under the same Gaussian perturbation.
\begin{corollary}\label{cor:core-gaussian}
Suppose \(r<\min(s,c)\) and \(E\) has independent standard Gaussian entries.
Then, as \(\varepsilon\to0\),
\[
 \E\|\Phi_r(M+\varepsilon E)-M\|_F^2
 =\varepsilon^2\risk(S,P)+O(|\varepsilon|^3).
\]
\end{corollary}

\begin{corollary}\label{cor:post-gaussian}
Suppose \(r<\min(s,c)\) and at least one row and one column remain unsampled.
If \(E\) has independent standard Gaussian entries, then \(\Psi_r'(M;E)\) exists almost surely. For each \(\varepsilon>0\), the two moments
\begin{equation}\label{eq:post-gaussian-limit}
 \E\|\Psi_r'(M;E)\|_F^2,\qquad
 \E\|\Psi_r(M+\varepsilon E)-M\|_F^2
\end{equation}
are finite if \(|s-c|\ge2\). If \(|s-c|\le1\), the mean-square recovery error is infinite,
and the derivative has infinite second moment exactly when at least one of
\((I-SS^\top)U\) and \((I-PP^\top)V\) is nonzero.
\end{corollary}
Here \((S^\top A P)^\dagger\) in \(\Psi_r\) denotes the exact Moore--Penrose inverse of the sampled intersection.
Both proofs are in Appendix~\ref{app:risk-quantitative}.

\section{Numerical experiments}
\label{sec:numerics}

\rev{We compare first-order predictions with the error in recovering $M$ from its noisy observation $A$. For independent standard Gaussian matrices $L\in\R^{m\times r}$ and $R\in\R^{n\times r}$, we set}
 \begin{equation}\label{eq:simulation-model}
M=\frac{LR^\top}{\|LR^\top\|_F},\qquad A=M+\frac{\varepsilon}{\sqrt{mn}}E.
\end{equation} 
\rev{When $\E\|E\|_F^2=mn$, $\varepsilon$ is the relative root-mean-square (RMS) noise level. Unless stated otherwise, $m=1000$, $n=800$, $r=20$, and $s=c=40$. We consider two CUR reconstructions, $\Phi_r(A)$ and $\Psi_r(A)$, both computed with the true rank $r$.}

\rev{In the sampling comparisons, all patterns use the same $50$ noise draws for each signal and noise law. We pool the $N$ reconstructions for each method and report the mean and sample standard deviation across $20$ independent signals of the relative Frobenius root-mean-square error (RMSE)}
 \[
e_{\rm rel}=\frac{1}{\|M\|_F}
\left(\frac1N\sum_{\ell=1}^{N}\|\widehat M_\ell-M\|_F^2\right)^{1/2}.
\] 

\subsection{\rev{Directional recovery and relative linearization}}
\label{sec:directional-experiment}
\rev{For one signal $M$, we fix the $40$ row and $40$ column indices selected by L-DEIM~\cite{gidisu2022hybrid} from $U$ and $V$, respectively. Writing $\mathcal J=\mathcal I_{T_M}^{S,P}$, we measure the linearization error by}
 \[
d_F(\varepsilon,E)=\left\|\Phi_r\!\left(M+\frac{\varepsilon}{\sqrt{mn}}E\right)-M-\frac{\varepsilon}{\sqrt{mn}}\mathcal J(E)\right\|_F.
\] 
\rev{Along a fixed unit-Frobenius Gaussian direction, the first-order prediction tracks the recovery error and $d_F$ decays quadratically in $\varepsilon$ (Figure~\ref{fig:quadratic-residual}A), as predicted by~\eqref{eq:uniform-taylor}.}

\rev{For a Gaussian matrix $G$, set $E_\perp=(I-\mathcal P_{T_M})G$, $E_0=(I-\mathcal J)G$, and}
 \[
E_\alpha=\frac{(1-\alpha)E_\perp+\alpha E_0}
{\|(1-\alpha)E_\perp+\alpha E_0\|_F},\qquad 0\le\alpha\le1.
\] 
\rev{Since $E_0=E_\perp-\mathcal J(E_\perp)$, CUR and SVD have equal first-order errors at $\alpha=1/2$ (panel B). At $\varepsilon=10^{-2}$, SVD recovers $M$ to roundoff at $\alpha=0$, and CUR does so at $\alpha=1$ despite noise in both the sampled rows and columns. As $\alpha$ approaches $1$, $d_F$ decreases in proportion to $1-\alpha$ down to roundoff (panel C), in agreement with~\eqref{eq:mixed}.}

\begin{figure}[!htbp]
\centering
\includegraphics[width=.32\linewidth]{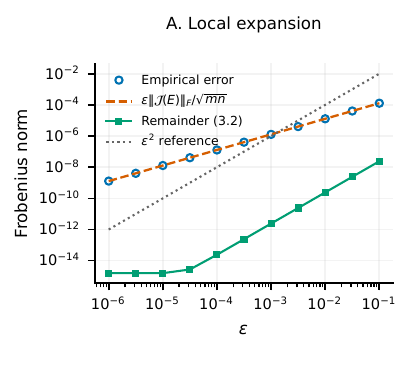}\hfill
\includegraphics[width=.32\linewidth]{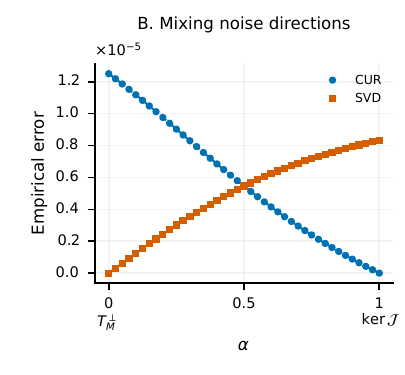}\hfill
\includegraphics[width=.32\linewidth]{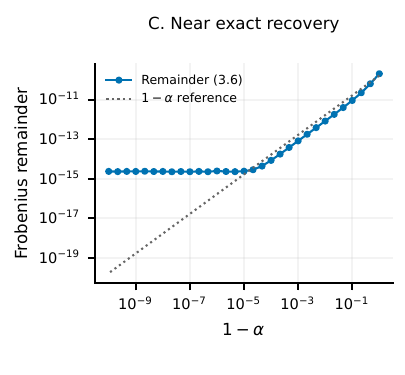}
\caption{\rev{Directional recovery. A: empirical error (circles), first-order prediction $\varepsilon\|\mathcal J(E)\|_F/\sqrt{mn}$ (dashed), $d_F$ (squares), and an $\varepsilon^2$ reference (dotted). B: CUR and SVD errors with first-order predictions. C: $d_F$ versus $1-\alpha$. Panels B--C use $\varepsilon=10^{-2}$.}}
\label{fig:quadratic-residual}
\end{figure}

\subsection{\rev{Comparison with CUR sampling methods}}
\label{sec:cur-sampling-comparison}
\rev{We recover $M$ with $\Phi_r(A)$ using uniform sampling, A-greedy, L-DEIM~\cite{gidisu2022hybrid}, rectangular MaxVol~\cite{mikhalev2018rectangular}, and leverage sampling without replacement~\cite{wang2017selection}. The four nonuniform methods use the true singular vectors. A-greedy selects its first $r$ rows and columns by column-pivoted QR of $U^\top$ and $V^\top$, respectively~\cite{drmac2016qdeim}, then adds indices by maximizing~\eqref{eq:row-gain} or its column analogue. We fix five patterns per randomized method and one per deterministic method. Each pattern uses all $50$ noise draws, giving $250$ or $50$ reconstructions per signal, contrast, and noise law.}

\rev{We give the $20\%$ of rows of $M$ with the largest Euclidean norms weight $\kappa$ and the rest weight $1$, and do the same for the columns. After scaling the row and column weights $v,w$ separately to mean one, we set $E_{ij}=\sqrt{v_iw_j}\,\xi_{ij}$ with i.i.d.\ Rademacher or standard Gaussian entries $\xi_{ij}$ independent of $M$. At $\varepsilon=10^{-2}$, A-greedy gives the lowest mean relative RMSE for $\kappa=1,2,4$. Uniform sampling gives the lowest for $\kappa=9,25$, as noise concentrates on the high-norm rows and columns that A-greedy favors (Figure~\ref{fig:variance-sweep}). The two noise laws have the same covariance given $M$ and hence the same first-order prediction. Their mean relative RMSEs differ by at most $0.0034$ percentage points.}

\begin{figure}[!htbp]
\centering
\includegraphics[width=.495\linewidth]{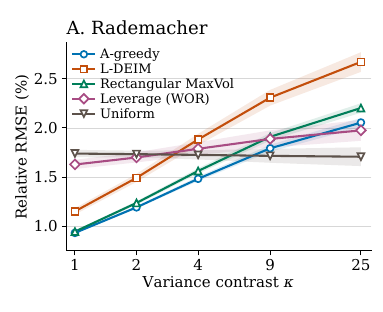}\hfill
\includegraphics[width=.495\linewidth]{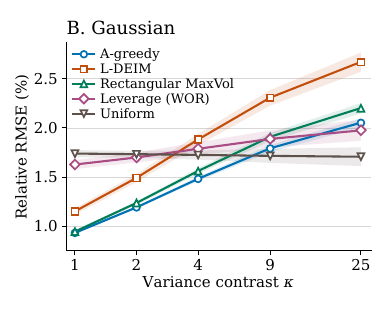}
\caption{\rev{Relative RMSE under Rademacher (A) and Gaussian (B) noise at $\varepsilon=10^{-2}$. Markers show empirical means, lines give first-order predictions, and bands span one sample standard deviation on either side of the mean across signals. WOR denotes sampling without replacement.}}
\label{fig:variance-sweep}
\end{figure}

\subsection{\rev{Evaluating sampling patterns from a noisy matrix}}
\label{sec:pilot-decisions}
\rev{To test how subspace estimation affects sampling, we use the leading $r$ singular vectors of}
 \[
A=M+\tau\sigma_r(M)\frac{G}{\|G\|_2},\qquad \tau\in\{10^{-10},10^{-9},\ldots,10^{-1},0.2,0.3,0.5\}.
\] 
We reuse three independent Gaussian directions $G$ per signal across all levels of the relative perturbation $\tau=\|A-M\|_2/\sigma_r(M)$. The five row and five column selectors give $25$ candidate pairs. Among these pairs, $\widehat d$ and $d_*$ minimize~\eqref{eq:risk} evaluated with the estimated and true subspaces, respectively.

At $\tau=0.5$, the mean relative discrepancy between the estimated and true coefficients is $3.91\%$ (Figure~\ref{fig:pilot-recovery}A). Across all $780$ fits, the relative selection loss
 \begin{equation}\label{eq:pilot-relative-selection}
\frac{\risk(\widehat d)-\risk(d_*)}{\risk(d_*)}
\end{equation} 
is at most $2.61\%$, even though $\tau=0.2,0.3,0.5$ lie outside the hypotheses of Proposition~\ref{prop:pilot} for all tested pairs. On $50$ fresh Rademacher draws per fit at $\varepsilon=10^{-4}$, sampling with $\widehat d$ or $d_*$ gives nearly identical relative RMSE in recovering $M$ (panel B).

\begin{figure}[!htbp]
\centering
\includegraphics[width=\textwidth]{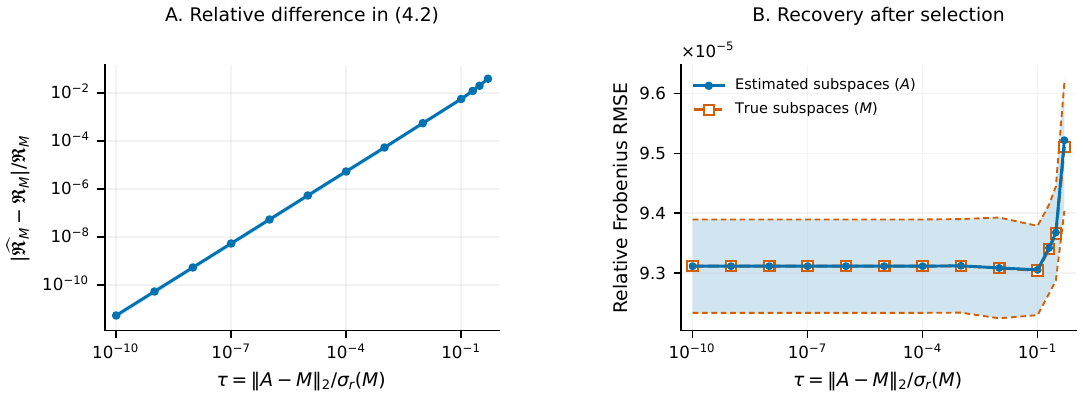}
\caption{Sampling evaluation using $A$ or $M$. A: relative discrepancy in~\eqref{eq:risk}, averaged over the $25$ pairs and three estimates per signal. B: relative recovery RMSE, averaging MSE over draws and estimates before taking square roots. Means and sample standard deviations are over $20$ signals. In B, the blue band ($A$) and orange dashed lines ($M$) show one standard deviation above and below each mean.}
\label{fig:pilot-recovery}
\end{figure}

\subsection{Comparing the two truncation orders}
\label{sec:truncation-experiment}
\rev{Figure~\ref{fig:truncation-order}A--B uses one signal, its $40$ highest-leverage rows and columns, and $A=M+tE$ with $\|E\|_F=1$. Along a fixed Gaussian direction, the first-order amplification is $1.19$ for $\Phi_r$ and $26.0$ for $\Psi_r$ (panel A). Panel B uses the normalized kernel directions $E_\eta$ from the proof of Corollary~\ref{cor:post-instability}, with $k=20$, $t=10^{-6}\eta$, and $x,y$ the normalized unsampled portions of the first columns of $U,V$. As $\eta$ decreases by decades from $10^{-1}$ to $10^{-10}$, the $\Psi_r$ error stays near $7\times10^{-7}$, so its amplification rises from $6.78$ to $7.05\times10^9$. Core truncation recovers $M$ to roundoff. We use the full core pseudoinverse for $\Psi_r$ in double precision, canceling $t^{-1}$ by block factorization in panels A--B.}

\rev{Panel C tracks the uncapped relative MSE over $20{,}000$ shared Gaussian draws at $\varepsilon=10^{-2}$ in~\eqref{eq:simulation-model}, with $c=40$ and $s=40,42,44$ obtained by adding rows in leverage order. For $\Phi_r$, Corollary~\ref{cor:core-gaussian} gives a finite small-noise MSE, and the final empirical MSEs range from $1.23\times10^{-4}$ to $1.43\times10^{-4}$. For $\Psi_r$, the population MSE is finite at $s-c=2,4$ and infinite at $s=c$ (Corollary~\ref{cor:post-gaussian}). In the square case the running mean diverges almost surely as $N\to\infty$, but it can appear stable until a rare large error occurs. Here draw $10{,}687$ alone accounts for $96\%$ of the total squared error.}

\begin{figure}[!htbp]
\centering
\includegraphics[height=.24\linewidth]{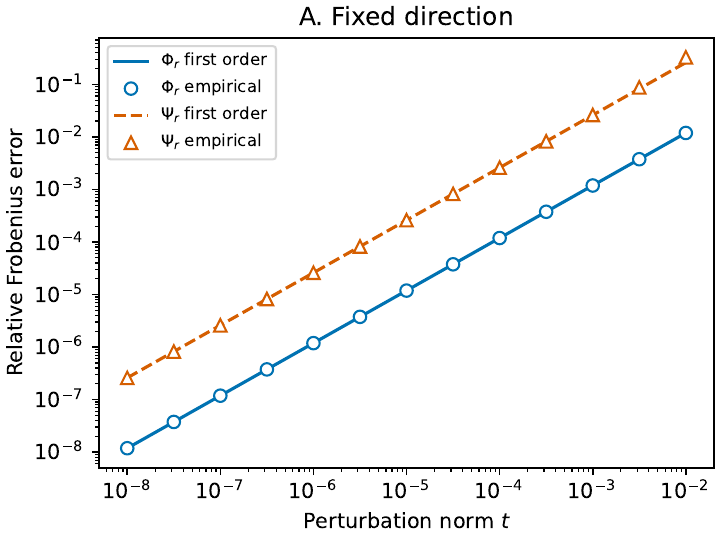}\hfill
\includegraphics[height=.24\linewidth]{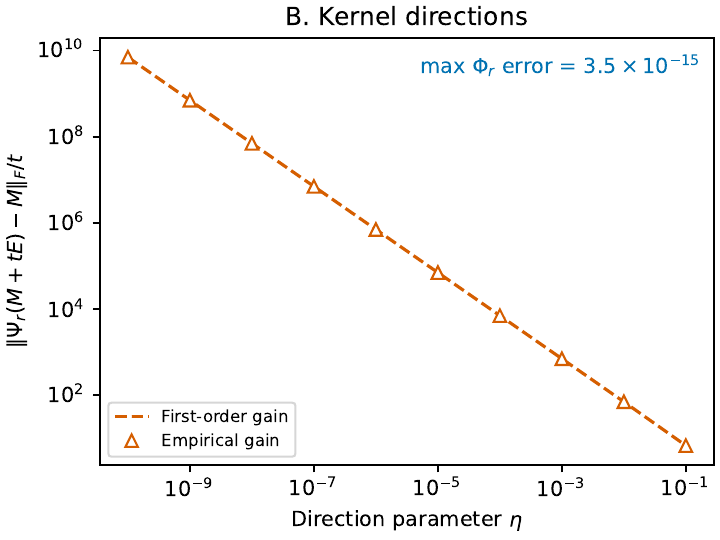}\hfill
\includegraphics[height=.24\linewidth]{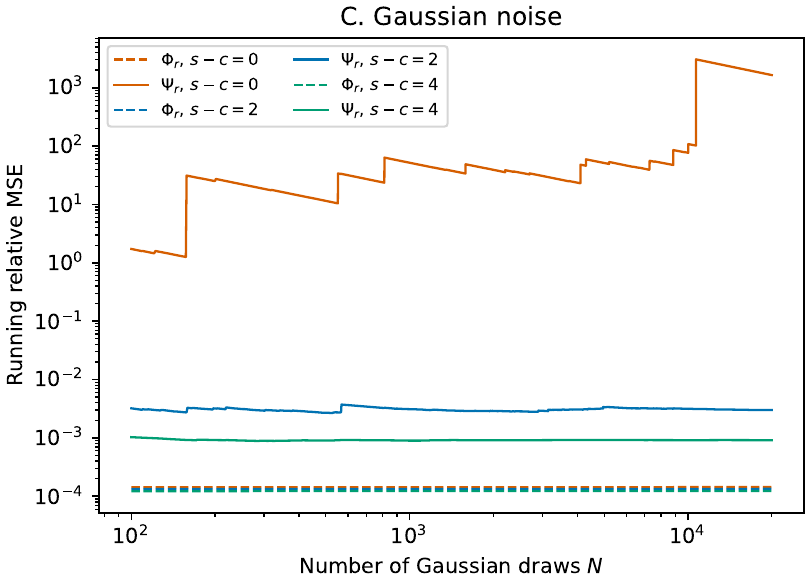}
\caption{\rev{Core truncation $\Phi_r$ and final truncation $\Psi_r$. A shows relative recovery errors (markers) and first-order predictions (lines). B shows the measured and predicted amplification for $\Psi_r$ along kernel directions. The maximum $\Phi_r$ error is $3.5\times10^{-15}$. C shows the running relative MSE for $\Psi_r$ (solid) and $\Phi_r$ (dashed) over Gaussian draws at $\varepsilon=10^{-2}$, with $s-c=0,2,4$.}}
\label{fig:truncation-order}
\end{figure}

\section{Conclusion}
\label{sec:conclusion}

For a fixed admissible row and column sampling pattern, we studied the rank-truncated CUR map near a rank-\(r\) matrix \(M\).
We showed that the Fr\'echet derivative of this map is the sampling-induced oblique tangent projector.
\rev{We characterized local exact recovery and gave quantitative remainder bounds. Under bounded centered isotropic noise, the leading mean-square error yields an A-optimal sampling criterion. Truncating after pseudoinversion can have unbounded directional amplification even when the perturbed intersection has rank strictly between \(r\) and \(\min(s,c)\). The resulting map can also be discontinuous.}
The numerical experiments confirm that the expansion predicts the leading first-order coefficient \rev{and illustrate sampling based on noisy subspace estimates}.
A natural next direction is to extend this local tangent-space viewpoint to tensor CUR and tensor cross approximations.

\section*{Acknowledgments}

\rev{This work was partially supported by NSF grant \mbox{DMS-2603463}. The author used ChatGPT to help edit the manuscript and check the proofs and code. The author assumes responsibility for all content.}


\appendix
\section{\rev{Supporting fixed-rank results}}

\label{app:auxiliary}
\rev{The classical fixed-rank pseudoinverse differential is recalled for reference, followed by the CUR annihilation identities.} \rev{The quantitative remainder lemma below is proved using an exact pseudoinverse perturbation identity.}

\begin{lemma} 
\label{lem:pinv-derivative}
Let \rev{$W\in\mathcal M_r$ and $H\in T_W\mathcal M_r$}. Then \rev{the differential of the Moore--Penrose pseudoinverse map on $\mathcal M_r$ is given by}
\[
\begin{aligned}
D(W^\dagger)[\rev{H}]
=&-W^\dagger \rev{H} W^\dagger
+W^\dagger W^{\dagger\top}\rev{H}^\top(I-WW^\dagger)+(I-W^\dagger W)\rev{H}^\top W^{\dagger\top}W^\dagger.
\end{aligned}
\]
\end{lemma}

\rev{The formula follows from \cite[Theorem~4.3]{golub1973differentiation}, expressed in the notation of the fixed-rank matrix manifold.}

\begin{lemma} 
\label{lem:cur-pinv-consequence}
Let \(M\in\mathbb R^{m\times n}\) have rank \(r\), let \((S,P)\) be admissible for \(M\), \rev{and set} \(W=S^\top MP\). Then
\(
(I-WW^\dagger)S^\top M=0,  MP(I-W^\dagger W)=0.
\)
Consequently, for any \rev{\(H\in T_W\mathcal M_r\)},
\(MP\,D(W^\dagger)[\rev{H}]\,S^\top M=-MPW^\dagger \rev{H} W^\dagger S^\top M.\)
\end{lemma}

\begin{proof}
\rev{With \(C=MP\) and \(R=S^\top M\), the projection identities in \cite[Lemma~5.1]{hamm2020perspectives} give \(C^\dagger C=W^\dagger W\) and \(RR^\dagger=WW^\dagger\), hence the two annihilation identities.} \rev{Substitution into Lemma~\ref{lem:pinv-derivative} eliminates its last two terms.}
\end{proof}

\begin{lemma} 
\label{lem:quantitative-pinv-remainder}
\rev{Let $X,Y\in\mathbb R^{s\times c}$ have the same rank $r\geq 1$, and set $\Delta=Y-X$ and $\gamma=\sigma_r(X)>0$.} \rev{Assume that, for some $\alpha\geq 1$, $\|Y^\dagger\|_2\le\alpha/\gamma$.} \rev{Let $H\in T_X\mathcal M_r$ and set $N=\Delta-H$.} \rev{Then \( Y^\dagger=X^\dagger+D(X^\dagger)[H]+\mathcal R_\dagger,\)
where}
 \[
\begin{aligned}
D(X^\dagger)[H]
={}&-X^\dagger H X^\dagger
+X^\dagger X^{\dagger\top}H^\top(I-XX^\dagger)+(I-X^\dagger X)H^\top X^{\dagger\top}X^\dagger,
\end{aligned}
\] 
\rev{and
\(
\|\mathcal R_\dagger\|\le(\alpha^3+3\alpha^2+4\alpha+3)\frac{\|\Delta\|_2\|\Delta\|}{\gamma^3}+\frac{3}{\gamma^2}\|N\|.
\)
Moreover,
\(\|Y^\dagger-X^\dagger\|\le(1+\alpha+\alpha^2)\frac{\|\Delta\|}{\gamma^2}.
\)}
\end{lemma}

\begin{proof}
\rev{The classical pseudoinverse perturbation identity \cite[Theorem~2.1]{wedin1973pseudo} gives} 
 \begin{equation}\label{eq:pinv-exact-perturbation}
\begin{aligned}
Y^\dagger-X^\dagger={}&-Y^\dagger \Delta X^\dagger
+ (Y^{\top}Y)^\dagger \Delta^\top(I-XX^\dagger)+(I-Y^\dagger Y)\Delta^\top (XX^{\top})^\dagger\\
={}&-Y^\dagger \Delta X^\dagger
+Y^\dagger Y^{\dagger\top}\Delta^\top(I-XX^\dagger)+(I-Y^\dagger Y)\Delta^\top X^{\dagger\top}X^\dagger,
\end{aligned}
\end{equation} 
\rev{where $\Delta=Y-X$.}
\rev{Using the ideal property, $\|X^\dagger\|_2=\frac{1}{\gamma}$, $\|Y^\dagger\|_2\le\frac{\alpha}\gamma$, and the fact that orthogonal projectors are contractions, we obtain}
 \begin{equation}\label{eq:pinv-explicit-lipschitz}
\|Y^\dagger-X^\dagger\|
\le(1+\alpha+\alpha^2)\frac{\|\Delta\|}{\gamma^2}.
\end{equation} 
\rev{Define $L_X[K]$ by replacing $H$ with $K$ in the formula for $D(X^\dagger)[H]$. Subtracting $L_X[H]$ from \eqref{eq:pinv-exact-perturbation} and using $H=\Delta-N$ gives}
 \begin{equation}\label{eq:pinv-remainder-decomposition}
\begin{aligned}
\mathcal R_\dagger={}&-(Y^\dagger-X^\dagger)\Delta X^\dagger+(Y^\dagger Y^{\dagger\top}-X^\dagger X^{\dagger\top})\Delta^\top(I-XX^\dagger)\\
&+(X^\dagger X-Y^\dagger Y)\Delta^\top X^{\dagger\top}X^\dagger+L_X[N].
\end{aligned}
\end{equation} 
\rev{The identity
\(Y^\dagger Y^{\dagger\top}-X^\dagger X^{\dagger\top}
=(Y^\dagger-X^\dagger)Y^{\dagger\top}+X^\dagger(Y^\dagger-X^\dagger)^\top\)
implies}
 \begin{equation}\label{eq:pinv-square-difference}
\|Y^\dagger Y^{\dagger\top}
-X^\dagger X^{\dagger\top}\|
\le\frac{\alpha+1}{\gamma}
\|Y^\dagger-X^\dagger\|.
\end{equation} 
\rev{For the third term, we use the row-space projector identity}
 \begin{equation}\label{eq:row-projector-difference}
\begin{aligned}
Y^\dagger Y-X^\dagger X
={}&Y^\dagger Y(I-X^\dagger X)
 -(I-Y^\dagger Y)X^\dagger X\\
={}&Y^\dagger \Delta(I-X^\dagger X)
 +(I-Y^\dagger Y)\Delta^\top X^{\dagger\top},
\end{aligned}
\end{equation} 
\rev{where the second equality follows from $X(I-X^\dagger X)=0$, $(I-Y^\dagger Y)Y^\top=0$, and $X^\top X^{\dagger\top}=X^\dagger X$. Consequently,}
 \begin{equation}\label{eq:row-projector-bound}
\|Y^\dagger Y-X^\dagger X\|
\le\frac{\alpha+1}{\gamma}\|\Delta\|.
\end{equation} 
\rev{The formula for $L_X$ and the ideal property give $\|L_X[N]\|\le 3\gamma^{-2}\|N\|$.} \rev{Combining \eqref{eq:pinv-remainder-decomposition}, \eqref{eq:pinv-square-difference}, and \eqref{eq:row-projector-bound}, we obtain}
 \[
\begin{aligned}
\|\mathcal R_\dagger\|
\le{}&
\frac{\alpha+2}{\gamma}
\|Y^\dagger-X^\dagger\|\|\Delta\|_2+\frac{\alpha+1}{\gamma^3}\|\Delta\|_2\|\Delta\|
+\frac{3}{\gamma^2}\|N\|.
\end{aligned}
\] 
\rev{Applying \eqref{eq:pinv-explicit-lipschitz} and simplifying the coefficient yields}
 \[
\begin{aligned}
\|\mathcal R_\dagger\|
&\le
\bigl[(\alpha+2)(1+\alpha+\alpha^2)
+(\alpha+1)\bigr]
\frac{\|\Delta\|_2\|\Delta\|}{\gamma^3}
+\frac{3}{\gamma^2}\|N\|\\
&=
(\alpha^3+3\alpha^2+4\alpha+3)
\frac{\|\Delta\|_2\|\Delta\|}{\gamma^3}
+\frac{3}{\gamma^2}\|N\|.
\end{aligned}
\] 
\end{proof}

\section{Proof of the rank-truncation expansion}
\label{app:rank-truncation-proof}

\begin{proof}[Proof of Lemma~\ref{lem:rank-r-truncation-first-order}]
\rev{If \(r=\min(s,c)\), rank truncation is the identity and the remainder vanishes.} \rev{We may assume \(r<\min(s,c)\).} \rev{Let \(\delta=\|\Delta\|_2\), \(\gamma=\sigma_r(W)\),} \(P_Q=QQ^\top\), and \(P_Z=ZZ^\top\). \rev{Set \(\widehat\Delta=\widehat W-W\) and \(R_{\mathrm{svd}}=W+\Delta-\widehat W=\Delta-\widehat\Delta\).}
  Then we have \(
\sigma_r(W+\Delta)\ge\gamma-\delta>\delta
\ge\sigma_{r+1}(W+\Delta),
\|R_{\mathrm{svd}}\|_2=\sigma_{r+1}(W+\Delta)\le\delta.
\)  
\rev{The Eckart--Young--Mirsky theorem also gives \(\|R_{\mathrm{svd}}\|\le\|W+\Delta-W\|=\|\Delta\|\), since \(W\) has rank \(r\).}

\rev{Write} \(\widehat W=\widehat Q\widehat\Lambda\widehat Z^\top\) \rev{and set} \(\widehat P_Q=\widehat Q\widehat Q^\top\), \(\widehat P_Z=\widehat Z\widehat Z^\top\). \rev{The singular-vector equations for \(W+\Delta\) give
\((W+\Delta)\widehat Z=\widehat Q\widehat\Lambda\),
\((W+\Delta)^\top\widehat Q=\widehat Z\widehat\Lambda\).}
\rev{Since \((I-P_Q)W=0\) and \((I-P_Z)W^\top=0\), we obtain}
 \[
\begin{aligned}
(I-P_Q)\widehat Q&=(I-P_Q)(W+\Delta)\widehat{Z}\widehat{\Lambda}^{-1}
=(I-P_Q)\Delta\widehat Z\widehat\Lambda^{-1},\\
(I-P_Z)\widehat Z&=(I-P_Z)\Delta^\top\widehat Q\widehat\Lambda^{-1}.
\end{aligned}
\] 
\rev{By Weyl's inequality, \(\|\widehat\Lambda^{-1}\|_2\le(\gamma-\delta)^{-1}\).} \rev{Since \(P_Q\) and \(\widehat P_Q\) are orthogonal projectors of the same rank, the principal-angle identity \cite[Theorem~I.5.5]{stewart1990matrix} gives \(\|P_Q-\widehat P_Q\|_2=\|(I-P_Q)\widehat Q\|_2\), and similarly for \(P_Z\) and \(\widehat P_Z\). Consequently,}
 \begin{equation}\label{eq:tsvd-subspace-bound}
\|P_Q-\widehat P_Q\|_2
\le\frac{\delta}{\gamma-\delta},
\qquad
\|P_Z-\widehat P_Z\|_2
\le\frac{\delta}{\gamma-\delta}.
\end{equation} 

\rev{Since \(\widehat W=(W+\Delta)\widehat P_Z\), we have}
 \[
\begin{aligned}
&(I-\mathcal P_{T_W})\widehat\Delta
=(I-P_Q)\widehat\Delta(I-P_Z)=(I-P_Q)(\widehat{W}-W)(I-P_Z)\\
=&(I-P_Q)\widehat{W}(I-P_Z)=(I-P_Q)(W+\Delta)\widehat{P}_Z(I-P_Z)=(I-P_Q)\Delta\widehat{P}_Z(I-P_Z)
\end{aligned}
\] 
\rev{Thus, by the ideal property and \eqref{eq:tsvd-subspace-bound},}
 \begin{equation}\label{eqn:bound_4_F}
\|(I-\mathcal P_{T_W})\widehat\Delta\|
\le\|\Delta\|\|\widehat{P}_Z(I-P_Z)\|_2=\|\Delta\|\|\widehat P_Z-P_Z\|_2
\le\frac{\delta\|\Delta\|}{\gamma-\delta}.
\end{equation} 
\rev{On the other hand, $R_{\mathrm{svd}}$ satisfies \(\widehat P_Q R_{\mathrm{svd}}=0\) and \(R_{\mathrm{svd}}\widehat P_Z=0\). Hence}
 \[
\begin{aligned}
\mathcal P_{T_W}R_{\mathrm{svd}}
&=P_Q R_{\mathrm{svd}}+(I-P_Q)R_{\mathrm{svd}}P_Z
=(P_Q-\widehat P_Q)R_{\mathrm{svd}}+(I-P_Q)R_{\mathrm{svd}}(P_Z-\widehat P_Z).
\end{aligned}
\] 
\rev{Using the ideal property, \eqref{eq:tsvd-subspace-bound}, and \(\|R_{\mathrm{svd}}\|\le\|\Delta\|\), we obtain}
 \begin{equation}\label{eqn:bound4_R}
\|\mathcal P_{T_W}R_{\mathrm{svd}}\|
\le\frac{2\delta}{\gamma-\delta}\|R_{\mathrm{svd}}\|
\le\frac{2\delta\|\Delta\|}{\gamma-\delta}.
\end{equation} 
\rev{Finally, since \(R_{\mathrm{svd}}=\Delta-\widehat\Delta\), we have
\(
\widehat\Delta-\mathcal P_{T_W}\Delta
=(I-\mathcal P_{T_W})\widehat\Delta-\mathcal P_{T_W}R_{\mathrm{svd}}.
\)
Combining \eqref{eqn:bound_4_F} and \eqref{eqn:bound4_R} yields
\(
\begin{aligned}
\|\widehat\Delta-\mathcal P_{T_W}\Delta\|
&\le\frac{3\delta\|\Delta\|}{\gamma-\delta}\le\frac{3}{1-\theta}\frac{\delta\|\Delta\|}{\gamma}.
\end{aligned}
\)
Since \(\widehat\Delta=\widehat W-W\), this proves the claimed expansion with \(R_{\mathrm{tr}}(\Delta)=\widehat\Delta-\mathcal P_{T_W}\Delta\).}
\end{proof}

\section{\rev{Constants and proof for the CUR remainder}}

\label{app:explicit-cur-proof}
\rev{For \(\theta\in(0,1/2)\), set \(\alpha_\theta=(1-\theta)^{-1}\).} \rev{The coefficients in Corollary~\ref{cor:quantitative-cur-expansion} are}
 \begin{equation}\label{eq:cur-remainder-coefficients}
c_1(\theta)=4\alpha_\theta^3+12\alpha_\theta^2+25\alpha_\theta+12,\qquad c_2(\theta)=2(1+\alpha_\theta+\alpha_\theta^2),\qquad c_3(\theta)=\alpha_\theta.
\end{equation}

\begin{proof}[\rev{Proof of Corollary~\ref{cor:quantitative-cur-expansion}}]
\rev{Set \(W=S^\top MP\), \(\Delta=S^\top EP\), and \(\delta=\|\Delta\|_2\). Let \(\widetilde W=W+\Delta\), \(\widehat W=\widetilde W_r\), and \(H=\mathcal P_{T_W}\Delta\). Since \(\delta\le\theta\gamma\), Lemma~\ref{lem:rank-r-truncation-first-order} gives \(\widehat W=W+H+R_{\mathrm{tr}}(\Delta)\), with}
 \begin{equation}\label{eq:CUR-Rtr-bound}
\|R_{\mathrm{tr}}(\Delta)\|\le 3\alpha_\theta\frac{\|\Delta\|_2\|\Delta\|}{\gamma}.
\end{equation} 
\rev{Since \(W\) itself has rank \(r\), the Eckart--Young--Mirsky theorem gives}
 \begin{equation}\label{eq:W-hat-W-distance}
\|\widehat W-W\|_2\le \|\widehat W-\widetilde W\|_2+\|\Delta\|_2\le 2\delta,\quad
\|\widehat W-W\|\le \|\widehat W-\widetilde W\|+\|\Delta\|\le 2\|\Delta\|.
\end{equation} 
\rev{Moreover, \(  \sigma_r(\widehat W)=\sigma_r(\widetilde W)\ge\gamma-\delta\ge(1-\theta)\gamma. \) Consequently,}
 \begin{equation}\label{eq:W-hat-pinv-bound}
\|\widehat W^\dagger\|_2  \le\frac{\alpha_{\theta}}{\gamma}.
    \end{equation} 
\rev{Applying the spectral-norm case of the difference estimate in Lemma~\ref{lem:quantitative-pinv-remainder} with $X=W$, $Y=\widehat W$, $\widehat\Delta=\widehat W-W$, and the tangent direction $H=\mathcal P_{T_W}\Delta$, we obtain}
 \begin{equation}\label{eq:W-hat-pinv-difference}
\|\widehat W^\dagger-W^\dagger\|_2\le 2(1+\alpha_{\theta}+\alpha_{\theta}^2)\frac{\delta}{\gamma^2}.
\end{equation} 
\rev{Here $\widehat\Delta-H=R_{\mathrm{tr}}(\Delta)$.} \rev{Define $\mathcal R_\dagger:=\widehat W^\dagger-W^\dagger-D(W^\dagger)[H]$.} \rev{The remainder part of the same lemma gives directly}
 \begin{equation}\label{eq:R-dagger-explicit-bound}
\begin{aligned}
\|\mathcal R_\dagger\| &\le (\alpha_{\theta}^3+3\alpha_{\theta}^2+4\alpha_{\theta}+3)\frac{\|\widehat\Delta\|_2\|\widehat\Delta\|}{\gamma^3}+\frac{3}{\gamma^2}\|R_{\mathrm{tr}}(\Delta)\|\\
&\le \left[4(\alpha_{\theta}^3+3\alpha_{\theta}^2+4\alpha_{\theta}+3)+9\alpha_\theta\right]\frac{\delta\|\Delta\|}{\gamma^3}
=c_1(\theta)\frac{\delta\|\Delta\|}{\gamma^3}.
\end{aligned}
\end{equation} 

\rev{For brevity in the remaining algebra, write \( C=MP,    R=S^\top M,  C_E=EP,  R_E=S^\top E. \) Then}
\rev{\(\Phi_r(M+E) = (C+C_E)\widehat W^\dagger(R+R_E).\)}
\rev{The constant and first-order terms are \(  CW^\dagger R=M\) and \(  C_EW^\dagger R  +    CD(W^\dagger)[H]R + CW^\dagger R_E = \mathcal I_{T_M}^{S,P}(E), \) as shown in Theorem~\ref{thm:cur-derivative}.} \rev{Therefore the remainder has the exact representation}
\begin{equation}\label{eq:exact-CUR-remainder-decomposition}
\rev{\mathcal R(E) = C\mathcal R_\dagger R+ C_E(\widehat W^\dagger-W^\dagger)R+ C(\widehat W^\dagger-W^\dagger)R_E+ C_E\widehat W^\dagger R_E.}
\end{equation}
\rev{By the ideal property and \eqref{eq:R-dagger-explicit-bound}, the first term is bounded by \( \|C\mathcal R_\dagger R\| \le \frac{c_1(\theta) ab}{\gamma^3}\delta\|\Delta\|. \) Using \eqref{eq:W-hat-pinv-difference}, the next two terms satisfy}
 \[
\|C_E(\widehat W^\dagger-W^\dagger)R\|+\|C(\widehat W^\dagger-W^\dagger)R_E\|\le\frac{2(1+\alpha_{\theta}+\alpha_{\theta}^2)}{\gamma^2}\delta\left(b\|EP\|+a\|S^\top E\|\right).
\] 
\rev{Finally, \eqref{eq:W-hat-pinv-bound} and the ideal property applied with either sampled factor in the unitarily invariant norm give}
 \[
\|C_E\widehat W^\dagger R_E\|
\le \frac{\alpha_\theta}{\gamma}
\min\bigl\{\|EP\|\|S^\top E\|_2,\ \|EP\|_2\|S^\top E\|\bigr\}.
\] 
\rev{Combining these estimates and using \eqref{eq:cur-remainder-coefficients} proves \eqref{eq:structured-CUR-remainder}.}
\end{proof}

\rev{Since selection matrices are contractions, \(\delta,\|EP\|_2,\|S^\top E\|_2\le\|E\|_2\) and \(\|S^\top EP\|,\|EP\|,\|S^\top E\|\le\|E\|\). Thus \(\|E\|_2\le\theta\gamma\) implies the gap condition \(\delta\le\theta\gamma\), and \eqref{eq:structured-CUR-remainder} gives \eqref{eq:uniform-taylor} with \(\rho=\theta\gamma\) and}
\begin{equation}\label{eq:explicit-CUR-remainder-constant}
\rev{C_2=\frac{c_1(\theta)ab}{\gamma^3}+\frac{c_2(\theta)(a+b)}{\gamma^2}+\frac{c_3(\theta)}{\gamma}.}
\end{equation}

\section{\rev{Proofs for rank-truncated CUR variants}}
\label{app:post-truncation}
 
\begin{proof}[Proof of Theorem~\ref{thm:post-truncation}]
Fix \(E\) and set \(A(t)=M+tE\), \(\Delta=S^\top EP\),
\(W(t)=S^\top A(t)P=W+t\Delta\), \(C=MP\), \(R=S^\top M\), and
\(F(t)=A(t)PW(t)^\dagger S^\top A(t)\).
Put \(\mathcal Q=[Q,Q_\perp]\), \(\mathcal Z=[Z,Z_\perp]\), and partition
\(\mathcal Q^\top\Delta\mathcal Z=(\Delta_{ij})_{i,j=1}^2\).
Write \(k=\rank\Delta_{22}\) and set
\(\Lambda(t)=\Lambda+t\Delta_{11}\) and
\(H(t)=\Delta_{22}-t\Delta_{21}\Lambda(t)^{-1}\Delta_{12}\).
For sufficiently small \(t\), \(\Lambda(t)\) is invertible. Block elimination in
\[
 \mathcal Q^\top W(t)\mathcal Z
 =\begin{pmatrix}\Lambda(t)&t\Delta_{12}\\
                  t\Delta_{21}&t\Delta_{22}\end{pmatrix}
\]
gives \(\rank W(t)=r+\rank H(t)\) for \(t\ne0\).
By hypothesis and \(H(0)=\Delta_{22}\), \(\rank H(t)=k\) near zero.
With smooth orthonormal bases \(Q_H(t),Z_H(t)\) for its column and row spaces,
write \(H(t)=Q_H(t)T(t)Z_H(t)^\top\) with smooth invertible \(T(t)\in\R^{k\times k}\).
Omit these blocks if \(k=0\).

Define
\[
 X(t)=\begin{pmatrix}I_r&0\\t\Delta_{21}\Lambda(t)^{-1}&Q_H(t)\end{pmatrix},
 \qquad
 Y(t)=\begin{pmatrix}I_r&t\Lambda(t)^{-1}\Delta_{12}\\0&Z_H(t)^\top\end{pmatrix}.
\]
Then \(X(t)\) has full column rank, \(Y(t)\) has full row rank, and
\(\mathcal Q^\top W(t)\mathcal Z=X(t)\operatorname{diag}(\Lambda(t),tT(t))Y(t)\).
For sufficiently small \(t\ne0\), the full-rank factorization identity
\cite{benisrael2003generalized} gives
\[
 W(t)^\dagger
 =\mathcal ZY(t)^\dagger
   \operatorname{diag}(\Lambda(t)^{-1},t^{-1}T(t)^{-1})
   X(t)^\dagger\mathcal Q^\top.
\]
For \(t\ne0\), this gives the three-factor product
\[
 F(t)=\bigl(A(t)P\mathcal ZY(t)^\dagger\bigr)
 \operatorname{diag}\bigl(\Lambda(t)^{-1},t^{-1}T(t)^{-1}\bigr)
 \bigl(X(t)^\dagger\mathcal Q^\top S^\top A(t)\bigr).
\]
Admissibility gives
\begin{equation}\label{eq:post-complement-identities}
 CZ_\perp=0,\quad Q_\perp^\top R=0,\quad
 (S^\top U)^\dagger Q_\perp=0,\quad
 Z_\perp^\top(V^\top P)^\dagger=0.
\end{equation}
Indeed, the first two identities follow from Lemma~\ref{lem:cur-pinv-consequence},
and the last two follow from
\(\range(S^\top U)=\range(Q)\) and
\(\range(P^\top V)=\range(Z)\).
We now write the two outer factors explicitly. Since
\(C=CW^\dagger W\), \(R=WW^\dagger R\), and \(W(t)=W+t\Delta\),
\[
A(t)P=CW^\dagger W(t)+t(EP-CW^\dagger\Delta),\quad S^\top A(t)=W(t)W^\dagger R+t(S^\top E-\Delta W^\dagger R).
\]
Moreover, \(Y(t)Y(t)^\dagger=X(t)^\dagger X(t)=I_{r+k}\), and
\(W^\dagger=Z\Lambda^{-1}Q^\top\) together with the block forms of \(X(t),Y(t)\) gives
\[
 CW^\dagger\mathcal QX(t)=[CZ\Lambda^{-1},0],\qquad
 Y(t)\mathcal Z^\top W^\dagger R
 =\begin{pmatrix}\Lambda^{-1}Q^\top R\\0\end{pmatrix}.
\]
Substituting the factorization of \(W(t)\) therefore yields
\[
\begin{aligned}
 A(t)P\mathcal ZY(t)^\dagger
 &=CW^\dagger\mathcal QX(t)\operatorname{diag}(\Lambda(t),tT(t))
   +t(EP-CW^\dagger\Delta)\mathcal ZY(t)^\dagger\\
 &=[CZ\Lambda^{-1}\Lambda(t),0]
   +t(EP-CW^\dagger\Delta)\mathcal ZY(t)^\dagger,\\
 X(t)^\dagger\mathcal Q^\top S^\top A(t)
 &=\operatorname{diag}(\Lambda(t),tT(t))Y(t)\mathcal Z^\top W^\dagger R
   +tX(t)^\dagger\mathcal Q^\top(S^\top E-\Delta W^\dagger R)\\
 &=\begin{pmatrix}\Lambda(t)\Lambda^{-1}Q^\top R\\0\end{pmatrix}
   +tX(t)^\dagger\mathcal Q^\top(S^\top E-\Delta W^\dagger R).
\end{aligned}
\]
Substituting both identities into the three-factor product and multiplying
the blocks gives, for \(t\ne0\),
\[
\begin{aligned}
 F(t)={}&M+tCW^\dagger\Delta W^\dagger R+t[CZ\Lambda^{-1},0]X(t)^\dagger\mathcal Q^\top(S^\top E-\Delta W^\dagger R)\\
 &+t(EP-CW^\dagger\Delta)\mathcal ZY(t)^\dagger
     \begin{pmatrix}\Lambda^{-1}Q^\top R\\0\end{pmatrix}\\
 &+t(EP-CW^\dagger\Delta)\mathcal ZY(t)^\dagger
     \operatorname{diag}\bigl(t\Lambda(t)^{-1},T(t)^{-1}\bigr)
     X(t)^\dagger\mathcal Q^\top(S^\top E-\Delta W^\dagger R).
\end{aligned}
\]
Every term after \(M\) has a factor \(t\) with a smooth coefficient:
\(X(t),Y(t)\) have constant full rank, and \(\Lambda(t),T(t)\)
are smooth and invertible. Since \(F(0)=CW^\dagger R=M\),
this expression proves that \(F\) is smooth through zero.

We now compute \(F'(0)\) by evaluating these coefficients at zero.
For the fourth coefficient, the block forms of \(X(0)\) and \(Y(0)\) give
\[
\begin{aligned}
 \mathcal ZY(0)^\dagger\operatorname{diag}\bigl(0,T(0)^{-1}\bigr)
 X(0)^\dagger\mathcal Q^\top
 &=[Z,Z_\perp Z_H(0)]\operatorname{diag}\bigl(0,T(0)^{-1}\bigr)
   \begin{pmatrix}Q^\top\\Q_H(0)^\top Q_\perp^\top\end{pmatrix}\\
 &=Z_\perp Z_H(0)T(0)^{-1}Q_H(0)^\top Q_\perp^\top
   =Z_\perp\Delta_{22}^\dagger Q_\perp^\top.
\end{aligned}
\]
Here \(\Delta_{22}^\dagger=Z_H(0)T(0)^{-1}Q_H(0)^\top\).
Evaluating the remaining coefficients and writing
\(\Delta_{22}=Q_\perp^\top\Delta Z_\perp\), we obtain
\[
\begin{aligned}
 F'(0)={}&CW^\dagger\Delta W^\dagger R+CW^\dagger(S^\top E-\Delta W^\dagger R)+(EP-CW^\dagger\Delta)W^\dagger R\\
 &+(EP-CW^\dagger\Delta)Z_\perp
       (Q_\perp^\top\Delta Z_\perp)^\dagger
       Q_\perp^\top(S^\top E-\Delta W^\dagger R).
\end{aligned}
\]
The identities \(CW^\dagger S^\top=\Pi_U\) and \(PW^\dagger R=\Pi_V\) give
\[
 EP-CW^\dagger\Delta=(I-\Pi_U)EP,\qquad
 S^\top E-\Delta W^\dagger R=S^\top E(I-\Pi_V).
\]
Hence \(F'(0)=\mathcal I_{T_M}^{S,P}(E)+K(E)\).
Thus \(F(t)=M+t[\mathcal I_{T_M}^{S,P}(E)+K(E)]+O(t^2)\).

It remains to truncate the CUR product: \(\Psi_r(M+tE)=(F(t))_r\).
For sufficiently small \(|t|\), \(\|F(t)-M\|_2<\sigma_r(M)/4\).
Applying Lemma~\ref{lem:rank-r-truncation-first-order} at \(M\), with perturbation
\(F(t)-M\), gives
\[
\begin{aligned}
 \Psi_r(M+tE)
 &=M+\mathcal P_{T_M}(F(t)-M)+O(\|F(t)-M\|_2\|F(t)-M\|)\\
 &=M+t\bigl[\mathcal I_{T_M}^{S,P}(E)+\mathcal P_{T_M}K(E)\bigr]+O(t^2),
\end{aligned}
\]
since \(\mathcal P_{T_M}\mathcal I_{T_M}^{S,P}=\mathcal I_{T_M}^{S,P}\).
The spectral gap at \(M\) also makes rank-\(r\) truncation smooth near \(M\),
so \(t\mapsto\Psi_r(M+tE)\) is smooth through zero.

Finally, \(\Delta_{22}(\lambda E)^\dagger=\lambda^{-1}\Delta_{22}(E)^\dagger\)
gives homogeneity.
\end{proof}

\begin{proof}[Proof of Corollary~\ref{cor:post-instability}]
Fix \(1\le k\le\min(s,c)-r\).
Let \(G_0\in\mathbb R^{(s-r)\times(c-r)}\) have ones in its first \(k\)
diagonal positions and zeros elsewhere. Choose nonzero vectors \(x,y\) such that
\(S^\top x=0\) and \(P^\top y=0\), with \(U^\top x\ne0\) or
\(V^\top y\ne0\); the hypotheses ensure such a choice.
Let \(e_1\) denote the first coordinate vector of the appropriate size and,
for \(0<\eta\le1\), set
\[
 E_\eta=xe_1^\top Z_\perp^\top P^\top
       +SQ_\perp e_1y^\top+\eta SQ_\perp G_0Z_\perp^\top P^\top.
\]
Since \(S^\top x=0\) and \(y^\top P=0\), the three summands are supported,
respectively, on unsampled rows and sampled columns, sampled rows and
unsampled columns, and the sampled intersection. These sets of entries are
disjoint, so \(E_\eta\ne0\).
The matrices \(SQ_\perp\) and \(PZ_\perp\) have orthonormal columns, so
unitary invariance, normalization, and the triangle inequality give
\(0<\|E_\eta\|\le\|x\|_2+\|y\|_2+\eta\|G_0\|\).
Using \eqref{eq:post-complement-identities} and the definitions of
\(\Pi_U,\Pi_V\), together with \(S^\top x=0\) and \(y^\top P=0\), gives
\[
\begin{aligned}
 \Pi_UE_\eta
 &=U(S^\top U)^\dagger Q_\perp
       (e_1y^\top+\eta G_0Z_\perp^\top P^\top)=0,\\
 E_\eta\Pi_V
 &=(xe_1^\top+\eta SQ_\perp G_0)Z_\perp^\top
       (V^\top P)^\dagger V^\top=0.
\end{aligned}
\]
Finally, \(S^\top E_\eta P=\eta Q_\perp G_0Z_\perp^\top\), whence
\(\Delta_{22}(E_\eta)=\eta G_0\) and
\[
 [Q,Q_\perp]^\top S^\top(M+tE_\eta)P[Z,Z_\perp]
 =\begin{pmatrix}\Lambda&0\\0&t\eta G_0\end{pmatrix}.
\]
Thus \(\rank\Delta_{22}(E_\eta)=k\) and the perturbed intersection has
rank \(r+k\) for \(t\ne0\), so Theorem~\ref{thm:post-truncation} applies.
Using \(G_0^\dagger=G_0^\top\) and \(G_0G_0^\top e_1=e_1\) in
\eqref{eq:post-noise-term} gives \(K(E_\eta)=\eta^{-1}xy^\top+E_\eta\).
Since \(\mathcal I_{T_M}^{S,P}(E_\eta)=0\), the reverse triangle inequality gives
\[
 \|\Psi_r'(M;E_\eta)\|\ge
 \eta^{-1}\|\mathcal P_{T_M}(xy^\top)\|
             -\|\mathcal P_{T_M}E_\eta\|.
\]
Moreover, orthogonality of the tangent and normal projections gives
\[
 \|\mathcal P_{T_M}(xy^\top)\|_F^2
 =\|x\|_2^2\|y\|_2^2
 -\|(I-UU^\top)x\|_2^2\|(I-VV^\top)y\|_2^2>0.
\]
The inequality is strict because \(x,y\ne0\) and the choice of \(x,y\)
makes at least one of \(\|(I-UU^\top)x\|_2\le\|x\|_2\) and
\(\|(I-VV^\top)y\|_2\le\|y\|_2\) strict.
Thus \(\mathcal P_{T_M}(xy^\top)\ne0\), so
\(\|\mathcal P_{T_M}(xy^\top)\|>0\) for the chosen unitarily invariant norm.
Theorem~\ref{thm:mixed}\textup{(i)} gives exact recovery by \(\Phi_r\)
along these kernel directions. Since \(\mathcal P_{T_M}\) is linear,
\(\|\mathcal P_{T_M}E_\eta\|\) also remains bounded. Thus homogeneity gives
\[
 \lim_{t\to0}
 \frac{\|\Psi_r(M+tE_\eta/\|E_\eta\|)-M\|}{|t|}
 =\frac{\|\Psi_r'(M;E_\eta)\|}{\|E_\eta\|}\longrightarrow\infty
 \quad(\eta\downarrow0).
\]
For any \(L>0\), choosing \(\eta\) sufficiently small and setting
\(E=E_\eta/\|E_\eta\|\) proves the assertion. Since
\(\|\Psi_r'(M;E)\|\) is unbounded over these directions with \(\|E\|=1\),
\(\Psi_r\) has no uniform local linear error bound at \(M\) and is not
Fr\'echet differentiable there.
\end{proof}

\section{\rev{Bounds and proofs for mean-square recovery}}

\label{app:risk-quantitative}
\rev{Assume the hypotheses of Theorem~\ref{thm:risk}, in particular $\norm{Z}_F\le B$ almost surely for a deterministic $B<\infty$.} \rev{Set $\mu_4=\E\norm{Z}_F^4$.} \rev{If $0<|\varepsilon|B\le\rho$, where $\rho,C_2$ are the constants in \eqref{eq:uniform-taylor}, then}
\begin{equation}\label{eq:risk-remainder}
\rev{\left|\frac{\E\norm{\Phi_r(M+\varepsilon Z)-M}_F^2}{\varepsilon^2}-\risk(S,P)\right|\le{}2C_2|\varepsilon|\sqrt{\risk(S,P)\mu_4} +C_2^2\varepsilon^2\mu_4.}
\end{equation}
\rev{Indeed, write $\Phi_r(M+\varepsilon Z)-M=\varepsilon\mathcal I_{T_M}^{S,P}(Z)+\mathcal R(\varepsilon Z)$, where $\norm{\varepsilon Z}_2\le|\varepsilon|B\le\rho$ almost surely. The Frobenius specialization of~\eqref{eq:uniform-taylor} gives $\norm{\mathcal R(\varepsilon Z)}_F\le C_2\varepsilon^2\norm{Z}_F^2$.} \rev{Expand the squared norm, apply Cauchy--Schwarz to the mixed term, and use $\E\norm{\mathcal I_{T_M}^{S,P}(Z)}_F^2=\risk(S,P)$.} \rev{In particular, for any two fixed admissible patterns with distinct leading risks, the ordering of their expected nonlinear recovery errors agrees with the leading-risk ordering for all sufficiently small nonzero $\varepsilon$.}

\begin{proof}[Proof of Corollary~\ref{cor:core-gaussian}]
Without loss of generality, assume \(s\ge c\). Write \(A=M+\varepsilon E\)
and let \(\mathcal B=\{|\varepsilon|\|E\|_2>\rho\}\), where
\(\rho\) is the radius in \eqref{eq:uniform-taylor}. On \(\mathcal B^c\),
the Taylor remainder satisfies \(\|\mathcal R(\varepsilon E)\|_F\le C_2\varepsilon^2\|E\|_2\|E\|_F\).
Expanding the squared error and using \(|\varepsilon|\|E\|_2\le\rho\) gives
\[
 \left|\|\Phi_r(A)-M\|_F^2-\varepsilon^2\|\mathcal I_{T_M}^{S,P}(E)\|_F^2\right|
 \le C|\varepsilon|^3\|E\|_2\|E\|_F^2.
\]
The expectation of this difference in absolute value on \(\mathcal B^c\) is \(O(|\varepsilon|^3)\).

For \(D=S^\top AP=S^\top MP+\varepsilon S^\top EP\) and \(\varepsilon\ne0\),
\(\|(D_r)^\dagger\|_2=1/\sigma_r(D)\) almost surely. Thus
\(\|\Phi_r(A)\|_F\le\|A\|_F^2/\sigma_r(D)\), and H\"older's inequality gives
\[
 \E[\|\Phi_r(A)\|_F^2\mathbf1_{\mathcal B}]
 \le \underbrace{\bigl(\E\,\sigma_r(D)^{-3}\bigr)^{2/3}}_{\hypertarget{core-holder-I}{\mathrm{(I)}}}
     \,\underbrace{\bigl(\E[\|A\|_F^{12}\mathbf1_{\mathcal B}]\bigr)^{1/3}}_{\hypertarget{core-holder-II}{\mathrm{(II)}}}.
\]
For \hyperlink{core-holder-I}{\textup{(I)}}, let \(T\) collect orthonormal right singular vectors for the two smallest singular values of \(D\).
Since \(r<c\), \(\|DT\|_F\le\sqrt2\,\sigma_r(D)\). For \(|J|=2\), write \(D_J\) for columns \(J\) of \(D\) and \(T_J\) for rows \(J\) of \(T\).
Cauchy--Binet gives \(\sum_{|J|=2}\det(T_J)^2=1\), so \(|\det T_J|\ge\tbinom{c}{2}^{-1/2}\) for some \(J\).
Since \(\|T_J\|_2\le1\), \(\|T_J^{-1}\|_2\le\sqrt{\binom c2}\).
Let \(P_{-J}\) project orthogonally onto the span of the remaining columns \(D_{-J}\).
Applying \(I-P_{-J}\) to \(DT=D_JT_J+D_{-J}T_{-J}\) yields
\[
 \|(I-P_{-J})D_J\|_F=\|(I-P_{-J})DT\,T_J^{-1}\|_F\le\sqrt{c(c-1)}\,\sigma_r(D).
\]
Taking inverse cubes and including all pairs in the sum yields
\(\sigma_r(D)^{-3}\le[c(c-1)]^{3/2}\sum_{|J|=2}\|(I-P_{-J})D_J\|_F^{-3}\).

Now fix \(J\) and condition on \(D_{-J}\). Its column space has an orthogonal
complement of dimension \(s-c+2\ge2\) almost surely. Choose two orthonormal
vectors in that complement, and let \(z\in\mathbb R^4\) collect the coordinates
of the two columns of \(D_J/|\varepsilon|\) along these vectors.
Conditionally, \(z\) is Gaussian with identity covariance, and
\(\|(I-P_{-J})D_J\|_F\ge|\varepsilon|\|z\|_2\).
Its conditional density is at most \((2\pi)^{-2}\) for every conditional mean.
Using this bound inside the unit ball and \(\|z\|_2^{-3}\le1\) outside gives
\[
 \E[\|z\|_2^{-3}\mid D_{-J}]
 \le (2\pi)^{-2}\int_{\|x\|_2\le1}\|x\|_2^{-3}\,dx+1
 =\tfrac12\int_0^1 t^{3}t^{-3}\,dt+1=\tfrac32.
\]
Taking expectations and summing over the pairs gives
\(\E\,\sigma_r(D)^{-3}\le C|\varepsilon|^{-3}\).

For \hyperlink{core-holder-II}{\textup{(II)}}, \(\mathbf1_{\mathcal B}\le(|\varepsilon|\|E\|_2/\rho)^{18}\)
and \(\|A\|_F\le\|M\|_F+\|E\|_F\) for \(|\varepsilon|\le1\). Hence
\[
 \E[\|A\|_F^{12}\mathbf1_{\mathcal B}]
 \le (|\varepsilon|/\rho)^{18}
       \E[(\|M\|_F+\|E\|_F)^{12}\|E\|_2^{18}]=O(|\varepsilon|^{18}).
\]
The expectation on the right is independent of \(\varepsilon\) and finite because Gaussian matrices have finite moments of every positive order.
Thus \(\mathrm{(I)}=O(|\varepsilon|^{-2})\) and \(\mathrm{(II)}=O(|\varepsilon|^6)\), giving
\(\E[\|\Phi_r(A)\|_F^2\mathbf1_{\mathcal B}]=O(|\varepsilon|^4)\).
Since \(\mathcal I_{T_M}^{S,P}\) is linear, this estimate and the indicator bound give
\[
 \E\bigl[(\|\Phi_r(A)-M\|_F^2+\varepsilon^2\|\mathcal I_{T_M}^{S,P}(E)\|_F^2)\mathbf1_{\mathcal B}\bigr]=O(|\varepsilon|^4).
\]
The \(O(|\varepsilon|^3)\) estimate on \(\mathcal B^c\) and the covariance identity in the proof of Theorem~\ref{thm:risk} give the claimed expansion.
\end{proof}

\begin{proof}[Proof of Corollary~\ref{cor:post-gaussian}]
Without loss of generality, assume \(s\ge c\). Set \(\Delta=S^\top EP\).
The standard Gaussian block \(\Delta_{22}=Q_\perp^\top\Delta Z_\perp\) has full column rank almost surely, so
Theorem~\ref{thm:post-truncation} gives \(\Psi_r'(M;E)=\mathcal I_{T_M}^{S,P}(E)+\mathcal P_{T_M}K(E)\).
The covariance calculation in Theorem~\ref{thm:risk} gives \(\E\|\mathcal I_{T_M}^{S,P}(E)\|_F^2=\risk(S,P)<\infty\).
Hence \(\Psi_r'(M;E)\in L^2\) if and only if \(\mathcal P_{T_M}K(E)\in L^2\).

Let \(\bar S,\bar P\) select the unsampled rows and columns. In the following
orthogonal coordinates, \(G_1,G_2,G_3,G_4,\Delta_{22}\) are mutually independent
standard Gaussian blocks.
\[
 [SQ,SQ_\perp,\bar S]^\top E[PZ,PZ_\perp,\bar P]
 =\left[\begin{array}{cc|c}
 Q^\top\Delta Z & G_1 & Q^\top S^\top E\bar P\\
 G_2 & \Delta_{22} & G_3\\ \hline
 \bar S^\top EPZ & G_4 & \bar S^\top E\bar P
 \end{array}\right].
\]
The outer factors of \(K(E)=(I-\Pi_U)EPZ_\perp\Delta_{22}^\dagger
Q_\perp^\top S^\top E(I-\Pi_V)\) expand as
\[
\begin{aligned}
 (I-\Pi_U)EPZ_\perp
 &=(I-\Pi_U)SQG_1+(I-\Pi_U)SQ_\perp\Delta_{22}
   +(I-\Pi_U)\bar SG_4,\\
 Q_\perp^\top S^\top E(I-\Pi_V)
 &=G_2Z^\top P^\top(I-\Pi_V)+\Delta_{22}Z_\perp^\top P^\top(I-\Pi_V)
   +G_3\bar P^\top(I-\Pi_V).
\end{aligned}
\]
Conditional on \(\Delta_{22}\), the two factors are independent.
We compute the conditional second moment using \eqref{eq:matrix-second-moment} and
\(\|\mathcal P_{T_M}K\|_F^2 =\|U^\top K\|_F^2+\|(I-UU^\top)KV\|_F^2.\)
Since \(U^\top(I-\Pi_U)SQ_\perp=0\), the left factor of
\(U^\top K(E)\) depends only on \(G_1,G_4\).
Conditional on \(\Delta_{22}\), averaging first over \(G_1,G_4\)
and then over \(G_2,G_3\) gives
\[
\begin{aligned}
 &\E[\|U^\top K(E)\|_F^2\mid\Delta_{22}]\\
 =&\|U^\top(I-\Pi_U)\|_F^2\,
 \E\!\left[\left\|\Delta_{22}^\dagger
 (G_2Z^\top P^\top+\Delta_{22}Z_\perp^\top P^\top+G_3\bar P^\top)
 (I-\Pi_V)\right\|_F^2\,\middle|\,\Delta_{22}\right]\\
 =&\|U^\top(I-\Pi_U)\|_F^2
 \Bigl[\bigl(\|I-\Pi_V\|_F^2-\|Z_\perp^\top P^\top(I-\Pi_V)\|_F^2\bigr)
       \|\Delta_{22}^\dagger\|_F^2
       +\|Z_\perp^\top P^\top(I-\Pi_V)\|_F^2\Bigr]\\
 =&\|U^\top(I-\Pi_U)\|_F^2
 \bigl[(\|I-\Pi_V\|_F^2-c+r)\|\Delta_{22}^\dagger\|_F^2+c-r\bigr].
\end{aligned}
\]
Here we used \(\Delta_{22}^\dagger\Delta_{22}=I_{c-r}\).
Since \(Z_\perp^\top P^\top\Pi_V=0\) and \(PZ_\perp\) has orthonormal
columns,
\(\|Z_\perp^\top P^\top(I-\Pi_V)\|_F^2 =\|Z_\perp^\top P^\top\|_F^2=c-r.\)
Similarly, \(Z_\perp^\top P^\top(I-\Pi_V)V=0\), so the right factor of
\((I-UU^\top)K(E)V\) depends only on \(G_2,G_3\).
Conditional on \(\Delta_{22}\), averaging in the reverse order gives
\[
\begin{aligned}
 &\E[\|(I-UU^\top)K(E)V\|_F^2\mid\Delta_{22}]\\
 =&\|(I-\Pi_V)V\|_F^2\,
 \E\!\left[\left\|(I-UU^\top)(I-\Pi_U)
 (SQG_1+SQ_\perp\Delta_{22}+\bar SG_4)\Delta_{22}^\dagger
 \right\|_F^2\,\middle|\,\Delta_{22}\right]\\
 =&\|(I-\Pi_V)V\|_F^2
 \bigl[(\|(I-UU^\top)(I-\Pi_U)\|_F^2-s+r)
        \|\Delta_{22}^\dagger\|_F^2+c-r\bigr].
\end{aligned}
\]
The orthogonal coordinate matrices above combine the coefficient norms of
the independent Gaussian blocks. In particular,
\[
\begin{aligned}
 \|[PZ,\bar P]^\top(I-\Pi_V)\|_F^2
 &=\|I-\Pi_V\|_F^2-(c-r),\\
 \|(I-UU^\top)(I-\Pi_U)[SQ,\bar S]\|_F^2
 &=\|(I-UU^\top)(I-\Pi_U)\|_F^2-(s-r).
\end{aligned}
\]
For the second identity, \(\Pi_USQ_\perp=0\) and \(U^\top SQ_\perp=0\) give
\((I-UU^\top)(I-\Pi_U)SQ_\perp=SQ_\perp\).
The remaining term \(c-r\) follows from
\(\|\Delta_{22}\Delta_{22}^\dagger\|_F^2=c-r\).
All cross terms have zero conditional expectation by independence and centering.

Since \(\|\Pi_U\|_F^2=a_S\) and \(\|\Pi_V\|_F^2=b_P\), the projector identities give
\[
\begin{gathered}
 \|U^\top(I-\Pi_U)\|_F^2=r+a_S-2r=a_S-r,\quad \|(I-\Pi_V)V\|_F^2=r+b_P-2r=b_P-r,\\
 \|I-\Pi_V\|_F^2=n+b_P-2r,\quad \|(I-UU^\top)(I-\Pi_U)\|_F^2=\|I-UU^\top\|_F^2=m-r.
\end{gathered}
\]
Substitution and addition yield
\[
\begin{aligned}
 &\E[\|\mathcal P_{T_M}K(E)\|_F^2\mid\Delta_{22}]\\
 &\quad=\bigl[(a_S-r)(n-c+b_P-r)+(b_P-r)(m-s)\bigr]\|\Delta_{22}^\dagger\|_F^2+(a_S+b_P-2r)(c-r).
\end{aligned}
\]
Since \(G_S,H_P\preceq I\), \(a_S=r\) exactly when \((I-SS^\top)U=0\),
and \(b_P=r\) exactly when \((I-PP^\top)V=0\).
If \(a_S=b_P=r\), the conditional second moment above vanishes, so the
derivative has finite second moment. Otherwise, \(m-s,n-c>0\) make the
coefficient of \(\|\Delta_{22}^\dagger\|_F^2\) positive, and the derivative
has finite second moment exactly when \(\E\|\Delta_{22}^\dagger\|_F^2<\infty\).

To evaluate this moment, let \(P_{-j}\) project orthogonally onto the
span of all columns of \(\Delta_{22}\) except \(\Delta_{22}e_j\).
The negative second moment identity
\cite[proof of Lemma~A.4]{tao2010universality} gives
\[
 \|\Delta_{22}^\dagger\|_F^2
 =\tr\bigl((\Delta_{22}^\top\Delta_{22})^{-1}\bigr)
 =\sum_{j=1}^{c-r}\frac{1}{\|(I-P_{-j})\Delta_{22}e_j\|_2^2}.
\]
Conditional on the other columns, \(P_{-j}\) is fixed and
\(\Delta_{22}e_j\) remains standard Gaussian. The squared residual norm
has the \(\chi^2_{s-c+1}\) law \cite[Section~2]{siriteanu2016schur},
since the orthogonal complement has
dimension \((s-r)-(c-r-1)=s-c+1\). Its reciprocal mean is
\[
 \frac{1}{2^{(s-c+1)/2}\Gamma((s-c+1)/2)}
 \int_0^\infty t^{(s-c-3)/2}e^{-t/2}\,dt.
\]
This mean is infinite for \(s-c\le1\); for \(s-c\ge2\), it equals
\(\Gamma((s-c-1)/2)/(2\Gamma((s-c+1)/2))=1/(s-c-1)\).
Summing over the \(c-r\) columns gives
\begin{equation}\label{eq:post-gaussian-moment}
 \E\|\Delta_{22}^\dagger\|_F^2=
 \begin{cases}
 +\infty, & s-c\le1,\\
 \dfrac{c-r}{s-c-1}, & s-c\ge2.
 \end{cases}
\end{equation}
This proves the assertion for \(\E\|\Psi_r'(M;E)\|_F^2\). We turn to the recovery error.
Fix \(\varepsilon>0\) and write \(A=M+\varepsilon E\) with the sampled rows and columns first.
\[
 A=\begin{bmatrix}W+\varepsilon\Delta&Y\\X&*\end{bmatrix},\qquad
 F=\begin{bmatrix}
 W+\varepsilon\Delta & (W+\varepsilon\Delta)(W+\varepsilon\Delta)^\dagger Y\\
 X & X(W+\varepsilon\Delta)^\dagger Y
 \end{bmatrix},
\]
where \(W=S^\top MP\) and \(F=AP(W+\varepsilon\Delta)^\dagger S^\top A\).
The blocks \(W+\varepsilon\Delta,X,Y\) are independent shifted Gaussian
matrices, and the core has full column rank almost surely.
Given \(\Delta\), the cross terms from centering \(X,Y\) have mean zero.
Applying \eqref{eq:matrix-second-moment} twice to the centered parts gives
\[
\medmuskip=3mu
\begin{aligned}
 &\E[\|X(W+\varepsilon\Delta)^\dagger Y\|_F^2\mid\Delta]=\|(\E X)(W+\varepsilon\Delta)^\dagger(\E Y)\|_F^2
   +\varepsilon^2(m-s)\|(W+\varepsilon\Delta)^\dagger(\E Y)\|_F^2\\
 &\quad+\varepsilon^2(n-c)\|(\E X)(W+\varepsilon\Delta)^\dagger\|_F^2
   +\varepsilon^4(m-s)(n-c)\|(W+\varepsilon\Delta)^\dagger\|_F^2.
\end{aligned}
\]
Here \(\E X=\bar S^\top MP\) and \(\E Y=S^\top M\bar P\) are fixed.
All four terms are nonnegative, and each is bounded by a fixed multiple
of \(\|(W+\varepsilon\Delta)^\dagger\|_F^2\).

Apply the same column-distance identity to \(W+\varepsilon\Delta\).
Conditional on all but one column, the orthogonal residual is Gaussian
in \(s-c+1\) dimensions, with a possibly nonzero mean and covariance
\(\varepsilon^2I\) \cite[Section~2]{siriteanu2016schur}.
Its density is continuous and positive at zero, and bounded uniformly
in the conditional mean. The inverse-square integral near zero is
governed by \(\int_0^1 t^{s-c-2}\,dt\); outside the unit ball the inverse
squared norm is at most one. Thus each conditional inverse-square moment
is infinite for \(s-c\le1\) and uniformly bounded for \(s-c\ge2\).
Summing over the columns gives the same classification for
\(\E\|(W+\varepsilon\Delta)^\dagger\|_F^2\).

If \(s-c\le1\), the last term in the four-term identity gives
\(\E\|F\|_F^2=+\infty\). Since \(\rank F\le c\),
\(\frac{r}{c}\|F\|_F^2\le\|F_r\|_F^2 \le2\|F_r-M\|_F^2+2\|M\|_F^2,\)
so \(\E\|\Psi_r(A)-M\|_F^2=+\infty\).
If \(s-c\ge2\), the four-term identity gives a finite second moment for
the lower-right block. The other blocks have finite second moments since
\(W+\varepsilon\Delta,X,Y\) are Gaussian and
\((W+\varepsilon\Delta)(W+\varepsilon\Delta)^\dagger\) is an orthogonal
projector. Thus \(\E\|F\|_F^2<\infty\), and \(\|F_r\|_F\le\|F\|_F\)
gives finite mean-square recovery error.
At singular-value ties, use any fixed measurable truncated-SVD choice.
\end{proof}

\section{Sampling from a noisy matrix estimate}

\label{app:additional-estimates}
\rev{Assume that the full noisy matrix $A=M+E$ is available.} \rev{Its rank-$r$ truncated SVD supplies estimated signal subspaces, held fixed while candidate sampling patterns vary.} \rev{The following result gives an explicit noise budget for estimating the leading isotropic risk in~\eqref{eq:risk}.}

\begin{proposition}\label{prop:pilot}
\rev{Let $M=U\Sigma V^\top\in\mathbb R^{m\times n}$ be a compact SVD of a rank-$r$ matrix, with $1\le r<\min(m,n)$, and let $\|E\|_2\le\varepsilon$.} \rev{For an admissible candidate $(S,P)$, set}
\rev{\(\sigma=\sigma_r(M),\gamma=\min\{\sigma_r(S^\top U),\sigma_r(V^\top P)\}>0.\)}
\rev{Suppose $0\le\varepsilon<\sigma\gamma/(1+\gamma)$, and write the rank-$r$ truncated SVD of $A$ as $A_r=\widehat U\widehat\Sigma\widehat V^\top$. Define}
\rev{\(\delta=\frac{\varepsilon}{\sigma-\varepsilon},q=(1-\delta/\gamma)^{-2}-1.\)}
\rev{Then}
 \[
\max\{\|UU^\top-\widehat U\widehat U^\top\|_2,
        \|VV^\top-\widehat V\widehat V^\top\|_2\}\le\delta<\gamma.
\] 
\rev{Both sampled estimated bases have full rank, and}
\rev{\(\widehat a_S=\tr\bigl((\widehat U^\top SS^\top\widehat U)^{-1}\bigr),\widehat b_P=\tr\bigl((\widehat V^\top PP^\top\widehat V)^{-1}\bigr)\)}
\rev{satisfy $|\widehat a_S-a_S|\le q a_S$ and $|\widehat b_P-b_P|\le q b_P$.} \rev{Let $\riskhat(S,P)$ be~\eqref{eq:risk} evaluated at $\widehat a_S,\widehat b_P$. Then $|\riskhat(S,P)-\risk(S,P)|\le B(S,P)$, where}
 \begin{equation}\label{eq:pilot-error}
 B(S,P)=q\bigl[(n+b_P-2r)a_S+(m+a_S-2r)b_P\bigr]+q^2a_Sb_P.
\end{equation} 
\end{proposition}

\begin{proof}
\rev{Since $\gamma\le1$, the noise budget gives $\varepsilon<\sigma/2$ and $\delta<\gamma$.} \rev{Singular-value perturbation gives
$\sigma_r(A)\ge\sigma-\varepsilon>\varepsilon\ge\sigma_{r+1}(A)$, so the truncation cutoff is strict.} \rev{From $A\widehat V=\widehat U\widehat\Sigma$,}
\rev{\((I-UU^\top)\widehat U =(I-UU^\top)E\widehat V\widehat\Sigma^{-1}.\)}
\rev{The equal-rank projector identity~\cite{stewart1990matrix} therefore gives}
\rev{\(\|UU^\top-\widehat U\widehat U^\top\|_2 \le\frac{\varepsilon}{\sigma_r(A)}\le\delta.\)}
\rev{Since $U^\top U=I_r$,}
 \[
(S^\top UU^\top)(S^\top UU^\top)^\top
=S^\top UU^\top S=(S^\top U)(S^\top U)^\top.
\] 
\rev{Thus $S^\top UU^\top$ and $S^\top U$ have the same first $r$ singular values, and likewise for $\widehat U$.} \rev{Weyl's singular-value inequality and $\|S^\top\|_2=1$ therefore yield, for $j=1,\ldots,r$,}
 \[
|\sigma_j(S^\top\widehat U)-\sigma_j(S^\top U)| \le\|S^\top(\widehat U\widehat U^\top-UU^\top)\|_2 \le\|\widehat U\widehat U^\top-UU^\top\|_2\le\delta.
\] 
\rev{In particular, $\sigma_r(S^\top\widehat U)\ge\gamma-\delta>0$, so $S^\top\widehat U$ has full column rank.} \rev{For $x\ge\gamma$ and $|y-x|\le\delta$, monotonicity and convexity of $t\mapsto t^{-2}$ on $(0,\infty)$ give}
 \[
 |y^{-2}-x^{-2}|\le(x-\delta)^{-2}-x^{-2}
 =x^{-2}\bigl[(1-\delta/x)^{-2}-1\bigr].
\] 
\rev{Since $x\ge\gamma$, the bracket is at most $q$.}
\rev{Summing gives $|\widehat a_S-a_S|\le q\sum_{j=1}^r\sigma_j(S^\top U)^{-2}=q a_S$.}
\rev{Applying the same argument to $A^\top$ with selector $P$ gives the bounds for $V$, including $|\widehat b_P-b_P|\le q b_P$.}
\rev{With $\Delta a=\widehat a_S-a_S$ and $\Delta b=\widehat b_P-b_P$,}
 \[
\riskhat(S,P)-\risk(S,P) ={}(n+b_P-2r)\Delta a +(m+a_S-2r)\Delta b+\Delta a\Delta b.
\] 
\rev{Both linear coefficients are positive because $a_S,b_P\ge r$, proving~\eqref{eq:pilot-error}.}
\end{proof}

\rev{For two candidates satisfying the noise condition, the estimated ordering is correct if the true risk gap exceeds the sum of their bounds in~\eqref{eq:pilot-error}.}

\rev{Let $\widehat d$ and $d_*$ minimize $\riskhat$ and $\risk$, respectively, over a finite, possibly data-dependent set of pairs $d=(S,P)$ with well-defined scores. Then}
 \[
\medmuskip=3mu
 \risk(\widehat d)-\risk(d_*)=[\risk(\widehat d)-\riskhat(\widehat d)]+[\riskhat(\widehat d)-\riskhat(d_*)]+[\riskhat(d_*)-\risk(d_*)].
\] 
The middle term is nonpositive because $\widehat d$ minimizes $\riskhat$. If both minimizers satisfy the noise condition, \eqref{eq:pilot-error} bounds the first and last terms. Since $d_*$ minimizes $\risk$, we obtain \begin{equation}\label{eq:pilot-selection-bound}
\rev{0\le\risk(\widehat d)-\risk(d_*)\le B(\widehat d)+B(d_*).}
\end{equation}

\bibliography{ref}
\bibliographystyle{plain}
\end{document}